\documentclass[10pt]{amsart}

\usepackage[british]{babel}
\usepackage[utf8]{inputenc}
\usepackage[T1]{fontenc}
\usepackage{amsfonts}
\usepackage{amssymb}
\usepackage{amsmath}
\usepackage{amsthm}
\usepackage{latexsym}
\usepackage{color}
\usepackage{amscd}
\usepackage{graphicx}
\usepackage{tikz}
\usetikzlibrary{arrows.meta,decorations.pathreplacing,calc}
\definecolor{dfone}{RGB}{0,114,178}     
\definecolor{dftwo}{RGB}{213,94,0}      
\definecolor{dftheta}{RGB}{106,61,154}  
\definecolor{dfgrey}{RGB}{127,127,127}
\providecommand\figdata{data/}
\renewcommand\figdata{figures/data/}
\usepackage{mathtools}
\definecolor{paperlink}{RGB}{0,0,114}
\usepackage[colorlinks=true,linktoc=all,
pdfstartview=FitV,linkcolor=paperlink,citecolor=paperlink,
urlcolor=paperlink]{hyperref}
\usepackage{parskip}

\usepackage[margin=1.3in]{geometry}

\numberwithin{equation}{section}

\newtheorem{theorem}{Theorem}[section]
\newtheorem{proposition}[theorem]{Proposition}

\theoremstyle{definition}

\newtheorem{conjecture}[theorem]{Conjecture}

\theoremstyle{remark}
\newtheorem{remark}[theorem]{Remark}

\newcommand\N{{\mathbb N}}
\newcommand\R{{\mathbb R}}
\newcommand\T{{\mathbb T}}
\newcommand\C{{\mathbb C}}

\newcommand\Z{{\mathbb Z}}

\newcommand\dx{{\frac{\mathrm d}{\mathrm dx}}}

\newcommand{\dd}{{\, \mathrm d}}

\newcommand{\Id}{\hbox{Id}}

\let\oldmarginpar\marginpar
\renewcommand\marginpar[1]{\-\oldmarginpar[\raggedleft\footnotesize #1]%
{\raggedright\footnotesize #1}}

\newcommand{\norm}[1]{\left\| #1 \right\|}
\newcommand{\abs}[1]{\lvert#1\rvert}
\newcommand{\TV}{\mathrm{TV}}

\allowdisplaybreaks

\title[Döblin--Fourier cancellation]{Döblin--Fourier cancellation and kinetic Aleksandrov estimates}

\author[Amélie Loher]{Amélie Loher}
\address{Amélie Loher: All Souls College, University of Oxford.}
\email{amelie.loher@all-souls.ox.ac.uk}
\author[Connor Mooney]{Connor Mooney}
\address{Connor Mooney: Department of Mathematics, UC Irvine.}
\email{mooneycr@uci.edu}
\author[Clément Mouhot]{Clément Mouhot}
\address{Clément Mouhot: DPMMS, University of Cambridge.}
\email{c.mouhot@dpmms.cam.ac.uk}

\date{2 October 2026}

\begin{document}

\begin{abstract}
  We develop a cancellation mechanism for linear second-order kinetic equations in non-divergence form with measurable uniformly elliptic coefficients. We decompose the dynamics into spatial waves and follow two families of velocity paths along which the waves acquire nearly opposite phases. The parabolic Krylov--Safonov theory gives a common lower bound for the two velocity marginals. The corresponding contributions cancel up to a small phase error, producing a contraction \textit{à la Döblin}. Iterating this contraction yields exponential Fourier decay estimates with enhanced dissipation.
  We then present two applications.

 The first contribution is a kinetic Aleksandrov estimate: a maximum principle in which the source is measured in an $L^p$ norm. For rough coefficients $A(t,v)$ independent of position, we obtain the estimate for every $p>2n+1$, where $n$ is the dimension of position and velocity. The Fourier decay also gives spatial smoothness, and parabolic regularity gives H\"older continuity in time and velocity. For autonomous coefficients $a(x,v)$ in dimension one, position serves as time away from zero velocity. A Harnack comparison controls returns to small velocity intervals and yields the estimate for every $p>4$. Both thresholds are optimal among those valid for all ellipticity ratios.

The second contribution concerns position on the torus and velocity on the sphere, with measurable uniformly elliptic diffusion coefficients $A(t,v)$ independent of position. We prove enhanced dissipation with the optimal square-root frequency power and Gevrey regularity in position. For autonomous coefficients $A=A(v)$, we also obtain exponential convergence to equilibrium in $L^2$ weighted by the invariant velocity measure, and a quantitative spectral gap. To our knowledge, these are the first results on decay established for laws of kinetic equations with rough coefficients.

\end{abstract}

\maketitle

\tableofcontents

\section{Introduction}

For elliptic and parabolic equations in non-divergence form with measurable uniformly elliptic coefficients, the classical maximum principle bounds a sub-solution by its boundary values and the $L^\infty$ norm of the positive part of the source. Aleksandrov estimates replace this source norm by the $L^n$ and $L^{n+1}$ norms, respectively. The elliptic estimate grew out of Aleksandrov's normal-map argument and the work of Aleksandrov, Bakelman and Pucci~\cite{Aleksandrov1960,Bakelman1961,Pucci1966}; its parabolic counterpart is due to Krylov and Tso~\cite{Krylov1976,Krylov1979,Tso1985}. These estimates underlie the Harnack theory of Krylov and Safonov~\cite{KrylovSafonov}.

We consider classical sub-solutions $u(t,x,v)$ to the kinetic equation:
\[
  \mathcal L u:=\partial_tu+v\cdot\nabla_xu-\mbox{Tr}\bigl(A(t,x,v)D_v^2u\bigr)\leq f,
  \qquad
  \lambda \, \mbox{Id}_n\leq A(t,x,v)\leq\Lambda \, \mbox{Id}_n.
  \refstepcounter{equation}\tag{\theequation}
    \label{eq:kin}
\]
This is the non-divergence form of what is often called \emph{Kolmogorov equations with rough coefficients} in the recent literature. Diffusion acts only in velocity, and its effect on position must pass through transport. Here $A$ is symmetric and measurable, $0<\lambda\leq\Lambda$, and $n$ is the common dimension of position and velocity. By a classical sub-solution we mean $u\in C^{1,1,2}$ in the interior of the domain considered, with the inequality holding almost everywhere. Our primary motivation is to derive Aleksandrov estimates in the kinetic setting, namely, bounds for $u$ in terms of its boundary values and the $L^p$ norm of the positive part of the source. For sufficiently large $p$, such an estimate is proved in the preprint~\cite{HaoZhang2026} and announced in the forthcoming work~\cite{NiebelForthcoming}; see Subsection~\ref{ss:limit}.

\subsection{The Döblin--Fourier cancellation estimate}
\label{ss:df-intro}

Let us explain the new estimate at the core of the paper when $A=A(t,v)$. We denote by $K_{s,t}(x,v,w)$ the adjoint fundamental solution, from starting time and velocity $(s,v)$ to terminal time and velocity $(t,w)$, with position displacement $x$. Its Fourier transform in $x$, denoted by $\widehat K_{s,t}^\xi(v,\cdot)$, is a complex measure of the terminal velocity $w$. The estimate proved in Section~\ref{s.fourier} is
\[
  \Big\Vert\widehat K_{s,t}^\xi(v,\cdot)\Big\Vert_{\TV}
  \leq C\exp\bigl(-c(t-s)\lambda^{\frac{1}{3}}\abs\xi^{\frac{2}{3}}\bigr),
  \qquad 0\leq s<t\leq1,\quad v,\xi\in\R^n,
  \refstepcounter{equation}\tag{\theequation}
  \label{eq:DF}
\]
with $c,C>0$ depending only on $n,\frac{\Lambda}{\lambda}$. A similar estimate is proved for autonomous coefficients $a(x,v)$ in dimension one, after a change of variable that exchanges the roles of time and position on intervals where velocity stays away from zero. The factor $\lambda^{\frac{1}{3}}$ in~\eqref{eq:DF} exhibits the enhanced dissipation studied for constant diffusion by Bedrossian--Coti Zelati~\cite{BedrossianCotiZelati2017} and Coti Zelati--Gallay~\cite{CotiZelatiGallay2023}, here with merely measurable diffusion coefficients depending on time and velocity.

\subsection{The first result}
\label{ss:main-result}

Combining the D\"oblin--Fourier cancellation estimate with the parabolic Aleksandrov estimate for the zero Fourier mode gives the following maximum principles. Their thresholds are optimal among those that are uniform in the ellipticity ratios for the stated coefficient classes.


The class of equations~\eqref{eq:kin} is invariant under the kinetic scaling $(t,x,v)\mapsto (r^2t,r^3x,rv)$ for $r>0$ and the Galilean transformation $(t,x,v)\mapsto(t+t_0,x+x_0+tv_0,v+v_0)$. This motivates, for \(P_0:=(t_0,x_0,v_0)\), the definition of the kinetic cylinder
\[
Q_r^-(P_0):=\bigl\{t_0-r^2<t<t_0,\ \abs{v-v_0}<r,\ 
\abs{x-x_0-(t-t_0)v_0}<r^3\bigr\},
\]
and we denote its initial, velocity, and incoming spatial boundaries by \(\partial_kQ_r^-(P_0)\):
\[
  \begin{aligned}
  \partial_kQ_r^-(P_0):=\overline{Q_r^-(P_0)}\cap\Bigl(
    &\bigl\{t=t_0-r^2\bigr\}\cup\bigl\{|v-v_0|=r\bigr\}\\
    &\cup\bigl\{|x-x_0-(t-t_0)v_0|=r^3,\ (v-v_0)\cdot (x-x_0-(t-t_0)v_0)\leq0\bigr\}\Bigr).
  \end{aligned}
\]
When \(r=1\), we omit the subscript \(r\). The subclass of equations~\eqref{eq:kin} with \(A=A(t,v)\) enjoys the same invariances. In contrast, in dimension \(1\), the subclass with \(a=a(x,v)\) is invariant under the kinetic scaling but not under the full Galilean group; it is invariant only under translations in \(t\) and \(x\). This is why we retain a general base point in the statement below. Figure~\ref{fig:kinetic-scaling} illustrates the shape of \(Q_r^-(P_0)\).
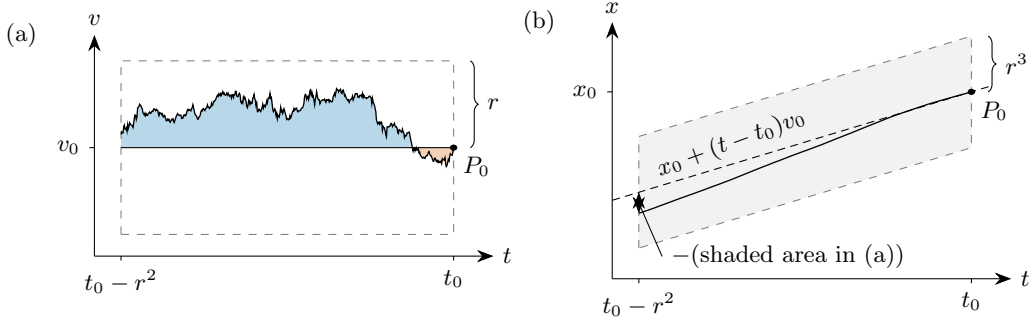
\begin{figure}[t]
  \centering
  \begingroup
  \tikzset{every picture/.append style={yscale=0.82}}
\begin{tikzpicture}[>={Stealth[length=2.2mm]},lab/.style={font=\small},note/.style={font=\small}]
  \def\vz{1.8}   
  \pgfmathsetmacro\shear{\vz*0.9/4.4}
  \begin{scope}[x=4.4cm,y=1.4cm]
    \begin{scope}
      \clip (-1,0) rectangle (0,1.2);
      \fill[dfone!28] plot file {\figdata f2_v.dat} -- (0,0) -- (-1,0) -- cycle;
    \end{scope}
    \begin{scope}
      \clip (-1,-1.2) rectangle (0,0);
      \fill[dftwo!28] plot file {\figdata f2_v.dat} -- (0,0) -- (-1,0) -- cycle;
    \end{scope}
    \draw[dfgrey,dashed] (-1,-1) rectangle (0,1);
    \draw[thin] (-1,0) -- (0,0);
    \draw[line width=0.5pt] plot file {\figdata f2_v.dat};
    \fill (0,0) circle (1.4pt) node[below right,note] {$P_0$};
    \draw[decorate,decoration={brace,amplitude=3pt}] (0.05,1) -- (0.05,0) node[midway,right=3pt,note] {$r$};
    \draw[->] (-1.08,-1.25) -- (0.12,-1.25) node[right,lab] {$t$};
    \draw (-1,-1.25) -- ++(0,-0.05) node[below,note] {$t_0-r^2$};
    \draw (0,-1.25) -- ++(0,-0.05) node[below,note] {$t_0$};
    \draw[->] (-1.08,-1.25) -- (-1.08,1.3) node[above,lab] {$v$};
    \draw (-1.08,0) -- ++(-0.015,0) node[left,note] {$v_0$};
    \node[lab] at (-1.3,1.3) {(a)};
  \end{scope}

  \begin{scope}[shift={(6.85cm,0.9cm)},x=4.4cm,y=0.9cm]
    \begin{scope}[cm={1,\shear,0,1,(0,0)}]
      \filldraw[fill=dfgrey!10,draw=dfgrey,dashed] (-1,-1) rectangle (0,1);
      \draw[thin,densely dashed] (-1.08,0) -- (0.06,0);
      \draw[line width=0.5pt] plot file {\figdata f2_x.dat};
      \draw[<->,thin] (-1,0) -- (-1,-0.3766);
      \coordinate (dev) at (-1,-0.19);
    \end{scope}
    \fill (0,0) circle (1.4pt) node[below right,note] {$P_0$};
    \draw[decorate,decoration={brace,amplitude=3pt}] (0.05,1) -- (0.05,0) node[midway,right=3pt,note] {$r^3$};
    \node[note,rotate=20,above] at (-0.7,-1.28) {$x_0+(t-t_0)v_0$};
    \draw[thin] (dev) -- (-0.93,-2.95) node[right,note] {$-$(shaded area in (a))};
    \draw[->] (-1.08,-3.35) -- (0.12,-3.35) node[right,lab] {$t$};
    \draw (-1,-3.35) -- ++(0,-0.08) node[below,note] {$t_0-r^2$};
    \draw (0,-3.35) -- ++(0,-0.08) node[below,note] {$t_0$};
    \draw[->] (-1.08,-3.35) -- (-1.08,1.25) node[above,lab] {$x$};
    \draw (-1.08,0) -- ++(-0.015,0) node[left,note] {$x_0$};
    \node[lab] at (-1.3,1.25) {(b)};
  \end{scope}
\end{tikzpicture}
  \endgroup
  \caption{\small{Why kinetic cylinders have scales $r^2$, $r$, and $r^3$. \textbf{(a)} A velocity path confined to $\abs{v-v_0}\leq r$ over a time interval of length $r^2$. \textbf{(b)} The corresponding position path ending at $P_0$. Its deviation from the straight motion $x_0+(t-t_0)v_0$ is the negative integral of $v-v_0$, and is therefore bounded by $r^3$.}}
  \label{fig:kinetic-scaling}
\end{figure}

\begin{theorem}[Sharp kinetic Aleksandrov estimate]
  \label{thm:aleksandrov-sharp}
  For every real-valued Borel measurable function \(f\), any classical subsolution \(u\in C(\overline{Q^-(P_0)})\cap C^{1,1,2}(Q^-(P_0))\) of~\eqref{eq:kin}, with \(f_+\in L^p(Q^-(P_0))\), satisfies
  \[
    \sup_{\overline{Q^-(P_0)}}u
    \leq \sup_{\partial_k Q^-(P_0)}u_+
    +C_{n,\lambda,\Lambda,p}\,\|f_+\|_{L^p(Q^-(P_0))}
    \refstepcounter{equation}\tag{\theequation}
    \label{eq:abp-kin}
  \]
  provided either \(A=A(t,v)\) and \(p>2n+1\), or \(n=1\), \(A=a(x,v)\), and \(p>4\). Here, \(C_{n,\lambda,\Lambda,p}>0\) is a constructive constant depending only on \(n\), \(\lambda\), \(\Lambda\), and \(p\).
\end{theorem}

The case \(A=A(t,v)\) is established in Section~\ref{sec:amp}. The condition \(p>2n+1\) is optimal and independent of the ellipticity ratio $\frac\Lambda\lambda$. In addition, the Döblin–Fourier cancellation estimate yields interior spatial smoothness, as shown in Section~\ref{sec:space}, see~\eqref{e:x-gradient}; parabolic regularity then implies Hölder continuity in all variables, as proved in Section~\ref{sec:holder}, see~\eqref{e:holder}.

The case $a(x,v)$ in dimension one is established in Section~\ref{s:spatial-roughness}. No fixed exponent below four works for all ellipticity ratios; see Section~\ref{ss:scaling-obstructions}. For each fixed ratio $\kappa=\frac\Lambda\lambda$, however, a further argument using a homogeneous Bellman barrier in \cite{Companion} gives the estimate for every $p>p_\ast(\kappa)$, where
\[
  p_\ast(\kappa)=1+\beta_*(\kappa)\in[3,4).
\]
Here $\beta_*$ is the homogeneous adjoint exponent defined in~\cite[Appendix~B]{Companion}. In particular, the estimate includes $p=4$. The proof is in~\cite[Theorem~1.2]{Companion}; neither optimality of $p_\ast$ nor an estimate at $p=p_\ast$ is claimed. We construct a homogeneous barrier satisfying a differential inequality for every admissible diffusion coefficient; taking the minimum over those coefficients gives a stationary Bellman inequality. Together with the Harnack comparison, this barrier controls concentration in small position and velocity boxes and makes the estimates summable for some $p<4$. The construction and proof are given in~\cite{Companion}.

\subsection{The limitation}
\label{ss:limit}

Theorem~\ref{thm:aleksandrov-sharp} treats coefficients $A(t,v)$ in any dimension and autonomous coefficients $a(x,v)$ in dimension one. It does not cover general coefficients $A(t,x,v)$. The corresponding Krylov--Safonov problem was posed in Silvestre's survey~\cite[Conjecture~8.1]{Silvestre2022}. The missing step in this programme was the following kinetic Aleksandrov estimate:
\begin{conjecture}[Kinetic Aleksandrov estimate]
\label{conjecture1}
There is a finite $p_0=p_0(n,\lambda,\Lambda)>1$ such that, for every $p>p_0$, every classical sub-solution $u\in C(\overline Q)$ of~\eqref{eq:kin}, with $f_+\in L^p(Q)$, satisfies
\[
    \sup_{\overline Q}u
    \leq \sup_{\partial_kQ}u_++C\norm{f_+}_{L^p(Q)},
    \refstepcounter{equation}\tag{\theequation}
    \label{eq:abp-kin-general}
\]
where $\partial_kQ$ consists of the initial, diffusive and incoming transport boundary, excluding the relative interior of the final-time face, and $C$ depends only on $n,\lambda,\Lambda,p$ and $Q$.
\end{conjecture}
Abedin and Tralli prove Harnack inequalities under a Cordes--Landis condition or a uniform modulus of continuity~\cite{AbedinTralli}. Dong and Yastrzhembskiy obtain global $L^p$ estimates when $A$ is measurable in $t$ and has vanishing mean oscillation in $(x,v)$~\cite{DongYastrzhembskiy}. Zhu proves the stationary one-dimensional Hölder estimate using parabolic theory away from $v=0$ and a barrier near $v=0$~\cite{Zhu2026}. A proof of Conjecture~\ref{conjecture1}, produced by GPT-6 Astra, is announced in a collective work still in preparation~\cite{NiebelForthcoming}. The preprint of Hao and Zhang~\cite{HaoZhang2026} gives a similar AI-assisted proof and resolves Conjecture~\ref{conjecture1}. In both results, the Lebesgue exponent threshold tends to infinity with the ellipticity ratio.

The Fourier method used in the paper relies on the decoupling of waves along a transported variable, e.g. $x$ in the case $A(t,v)$. In fact, no Fourier decay tending to zero can hold uniformly for general coefficients with fixed ellipticity bounds. We take $n=1$, write $g_t(v):=(4\pi t)^{-\frac{1}{2}}e^{-\frac{v^2}{4t}}$ and $q_N(v):=N^{-3}e^{-v^2}\cos(Nv)$ for $N\in\N$, and define on $[0,1]\times(\R/2\pi\Z)\times\R$
\begin{align*}
  u_N(t,x,v)&=g_{t+2}(v)+\varepsilon\sin(Nx)q_N'''(v),\\
  a_N(t,x,v)&=\frac{g_{t+2}(v)+\varepsilon N\cos(Nx)\bigl(vq_N'(v)-2q_N(v)\bigr)}{u_N(t,x,v)}.
\end{align*}
The identity $(vq_N'-2q_N)''=vq_N'''$ gives $\partial_tu_N+v\partial_xu_N=\partial_v^2(a_Nu_N)$. Moreover, $\frac{|q_N'''|+N|vq_N'-2q_N|}{g_{t+2}}$ is bounded uniformly in $t,v,N$, since its terms are polynomials times a faster-decaying Gaussian. Thus by choosing $\varepsilon>0$ small enough, we get $a_N$ smooth with $\frac{1}{2}\leq a_N\leq\frac{3}{2}$, and $u_N$ is a positive probability density with normalised measure in $x$ so that 
\[
  \norm{\hat u_N(1,N,\cdot)}_{L^1_v}
  =\frac{\varepsilon}{2}\norm{q_N'''}_{L^1_v}
  \longrightarrow\frac{\varepsilon}{\sqrt\pi}>0.
\]
In particular, no function $\eta(q)\to0$ can bound
\[
  \norm{\widehat u(1,N,\cdot)}_{L^1_v}\leq\eta(|N|)
\]
for every smooth coefficient with $\frac12\leq a\leq\frac32$ and every initial probability density $u(0,x,v)$ for the adjoint equation. The solutions $u_N$ above contradict such a bound. A bound of the form~\eqref{eq:DF}, uniform also in the starting position, would imply such an estimate by integration against $u(0,x,v)\dd x\dd v$, and is therefore impossible. Periodically extending $a_N$ and projecting the full-space evolution onto the torus gives the same obstruction on $\R$.

\subsection{The merit}

Although Conjecture~\ref{conjecture1} has now been resolved, there are two reasons for presenting Theorem~\ref{thm:aleksandrov-sharp} for the coefficient classes above.

First, the available estimates for general coefficients $A(t,x,v)$ require a Lebesgue exponent above a threshold that tends to infinity with the ellipticity ratio. This is unsatisfactory in comparison with the elliptic and parabolic Aleksandrov maximum principles. For $A=A(t,v)$, Theorem~\ref{thm:aleksandrov-sharp} reaches the scaling threshold $2n+1$, which is optimal for each fixed ellipticity ratio. For autonomous $a(x,v)$ in dimension one, the threshold $4$ is optimal among thresholds valid for all ellipticity ratios. The optimal threshold in $[3,4)$ as a function of the ratio remains open. We discuss the scaling obstructions and counter-examples in Subsection~\ref{ss:scalings}.

Second, the proof relies on a cancellation mechanism~\eqref{eq:DF} that is of independent interest. This contraction in total variation norm is reminiscent of Döblin’s argument~\cite{Doeblin1940}; here, however, the cancellation takes place in Fourier space, motivating the term \emph{Döblin--Fourier cancellation}. Further details are provided in Subsection~\ref{ss:strategy}. The resulting estimate captures \emph{enhanced dissipation} for kinetic equations with rough coefficients and yields quantitative spectral-gap (hypocoercivity) estimates in compact settings. To highlight its significance, we present in Subsection~\ref{ss:second} a second main result establishing mixing estimates for a non-divergence-form kinetic diffusion process with rough coefficients, spherical velocities, and a periodic spatial domain.

\subsection{Scaling obstruction, counter-examples and a conjecture}
\label{ss:scalings}

Let us explore more thoroughly the issue of critical exponents that is raised by Theorem \ref{thm:aleksandrov-sharp}. 

\subsubsection{The main scaling obstruction}

To see why a smaller exponent than $2n+1$ is not possible in the kinetic Aleksandrov estimate, we take a constant matrix $A$ and a smooth localized bump $\tilde u$ of height one. We shrink the bump by setting $u_r(t,x,v)=\tilde u(r^{-2}t,r^{-3}x,r^{-1}v)$. Its height stays one, while its associated source scales as follows by a straightforward change of variables:
\[
  \norm{\mathcal Lu_r}_{L^p} =r^{\frac{4n+2}{p}-2}\norm{\mathcal L\tilde u}_{L^p}.
\]
Hence if $p<2n+1$, this norm tends to zero as $r$ tends to zero. The height of the bump therefore cannot be bounded by its source norm on a fixed cylinder (note that the support of $u_r$ shrinks with $r$). At $p=2n+1$, a single shrinking bump no longer gives a contradiction. Instead, we add nested bumps so that their heights diverge while their source norms remain bounded; see Section~\ref{ss:scaling-obstructions}. Thus even in the case of constant coefficients the kinetic Aleksandrov estimate requires $p>2n+1$. 

\subsubsection{What we learn from the elliptic and parabolic theories}

Analogous scaling arguments yield the lower bounds \(\frac{n}{2}\) and \(1+\frac{n}{2}\) in the elliptic and parabolic cases, respectively. But, if scaling gives a necessary threshold, the rough coefficients may require a larger exponent. The elliptic and parabolic precedents illustrate the distinction.

We write $\kappa:=\frac{\Lambda}{\lambda}$ for the ellipticity ratio. The known lower bounds for the critical exponents are, respectively,
\[
  \max\left\{\frac{n}{2},\frac{n\kappa}{\kappa+n-1}\right\},
  \qquad
  \max\left\{1+\frac{n}{2},\frac{n\kappa}{\kappa+n-1}\right\}.
\]
The elliptic comparison is for $n\geq2$. The first terms come from scaling and the second from the radial examples recalled in Section~\ref{ss:scaling-obstructions}. The elliptic lower bound is strictly below $n$ for every finite $\kappa$. Pucci conjectured that it is optimal; see~\cite{PucciExtremal1966,AstalaIwaniecMartin2006,Trudinger2020}. Reverse Hölder estimates show that some exponents strictly below $n$ in the elliptic case and $n+1$ in the parabolic case are allowed, with an improvement depending on $\kappa$~\cite{FabesStroock1984,Escauriaza1993,CrandallFokKocanSwiech1998,DongKitano2025}. The optimal elliptic threshold tends to $n$ as $\kappa\to\infty$; in dimension two the sharp dependence is known, and the strong estimate fails at the endpoint~\cite{AstalaIwaniecMartin2006}. For autonomous parabolic coefficients and $n\geq3$, the estimate holds for some $p\in(1+\frac{n}{2},n)$~\cite[Theorem~4.13]{KrylovLd}. More precisely, the cited theorem requires $p\geq d_0$ and $p>1+\frac{n}{2}$, where $d_0\in(\frac{n}{2},n)$ depends on the ellipticity bounds. The interval $(1+\frac{n}{2},n)$ is empty in dimensions one and two. Together with the lower bounds above, these estimates give the limiting autonomous threshold $\max\{1+\frac{n}{2},n\}$ as $\kappa\to\infty$.

A related obstruction occurs for non-divergence form equations on groups of Heisenberg type. Danielli, Garofalo and Nhieu~\cite{DanielliGarofaloNhieu2003} construct counterexamples showing that no fixed exponent below the homogeneous dimension $Q$ can give the maximum estimate for all ellipticity ratios. Here the coefficient bounds may depend on the exponent. The kinetic homogeneous dimension is $2+4n$; the comparison motivates the discussion below, but does not by itself determine the kinetic threshold.

These comparisons motivate the following conjecture in the kinetic case.

\subsubsection{The sharp exponent in the kinetic case}

We conjecture that, in the kinetic case, the answer varies with the variables on which $A$ depends, as follows:
\begin{conjecture}[The critical kinetic Aleksandrov exponent]
  \label{conjecture}
  The estimate~\eqref{eq:abp-kin-general} of Conjecture~\ref{conjecture1} holds for the following exponents, depending on $A$:
  \begin{itemize}
  \item[(a)] When $A=A(t,x)$, it holds for every $p>1+2n$, independently of the ellipticity ratio. 
  \item[(b)] When $A=A(t,v)$, it holds for every $p>1+2n$, independently of the ellipticity ratio.
  \item[(c)] When $A=A(t,x,v)$, it holds for every $p>p_c$, where $p_c\in[1+2n,4n)$ depends on the ellipticity ratio and goes to $4n$ as the ratio tends to infinity.
  \end{itemize}
\end{conjecture}

The value $4n$ is the homogeneous dimension of $(x,v)$, i.e. $3n+n$; it is conjectured to be the lower bound for an exponent independent of the ellipticity ratio by analogy with the parabolic Aleksandrov estimate. The value $1+2n$ is the scaling lower bound. As explained below, in cases (a) and (b) it is possible to go below $4n$ thanks respectively to the the divergence-form structure and our Fourier-Döblin cancellation estimate.


\subsubsection{What is known and what remains open}

Our result proves case (b), which reaches the optimal scaling lower bound, and gives the ellipticity-independent threshold $4$ in case (c) when $n=1$ and $A=A(x,v)$. Moreover, the autonomous counterexamples recorded in \cite{Companion} exclude any $p<4n$ independent of the ellipticity ratio in (c), for all $n \ge 1$, as explained in Section~\ref{ss:scaling-obstructions}. 

Case (a) has additional structure and is solved by the divergence-form theory: since $A(t,x)$ commutes with velocity derivatives, the equation is also in divergence form. Let us briefly recall how the divergence-form theory gives the sharp scaling range $p>1+2n$~\cite{PascucciPolidoro2003,GuerandMouhot2022}. Indeed, its fundamental solution has mass one and is bounded by $C(t-s)^{-2n}$. Its $L^q$ norm to the power $q$ is therefore bounded by $C(t-s)^{-2n(q-1)}$. Time integration, Duhamel's formula, H\"older's inequality and comparison give the estimate for $1<q<1+\frac{1}{2n}$, or $p>1+2n$ by duality. Case (b) also has additional structure, uncovered by this paper: the spatial smoothing following from the Fourier decay estimate reduces the problem to the parabolic setting in $(t,v)$, which permits the sharp kinetic range $p>1+2n$ down to the scaling obstruction (the optimal parabolic ellipticity-independent exponent $\max\{1+\frac{n}{2},n\}$ is always below this kinetic scaling lower bound).

For $n=1$ and $A(x,v)$, the companion paper \cite{Companion} also gives an ellipticity-dependent threshold strictly below four in case (c), as recorded after Theorem~\ref{thm:aleksandrov-sharp}. The optimal dependence on the ellipticity ratio remains open. In higher dimension and the case $A(t,x,v)$ where the coefficients depend on all variables, it remains open to prove that $4n$ is the ellipticity-independent lower bound and to obtain the exact dependency on the ellipticity ratio below $4n$.

\subsection{The second result}
\label{ss:second}

Let $n\geq2$, let $\mathbb S^{n-1}$ be the unit sphere, and let $\T^n=\R^n/(2\pi\Z^n)$. The spatial measure on the torus is normalised to have mass one. For a symmetric measurable tangent matrix $A(t,v)$ satisfying $\lambda\Id\leq A\leq\Lambda\Id$ on $T_v\mathbb S^{n-1}$, we consider
\begin{equation}
  \partial_tu+v\cdot\nabla_xu=A:\nabla_S^2u
  \quad\hbox{on }\T^n\times\mathbb S^{n-1},
  \label{e:spherical-equation}
\end{equation}
where $\nabla_S^2u$ is the covariant Hessian on the sphere. We write $K_{s,t}$ for a transition kernel associated with $v\cdot\nabla_x+A(t,v):\nabla_S^2$, as constructed in Section~\ref{s:mixing}, and $\widehat K_{s,t}^\xi$ for its spatial Fourier transform at $\xi\in\Z^n$. Note that frequencies are now discrete. 
\begin{theorem}[Mixing for rough spherical velocity diffusion]
  \label{thm:mixing}
  There are $C,c>0$, depending only on $n,\frac\Lambda\lambda$, such that
  \[
    \sup_{v\in \mathbb S^{n-1}}\norm{\widehat K_{s,t}^\xi(v,\cdot)}_{\TV}
    \leq C\exp\left(-c(t-s)
    \min\left\{\frac{\abs \xi^2}{\lambda},\lambda^{\frac{1}{2}}\abs \xi^{\frac{1}{2}}\right\}\right),
    \qquad s<t,\quad \xi \in\Z^n\setminus\{0\}.
    \refstepcounter{equation}\tag{\theequation}
    \label{e:spherical-DF}
  \]
  In particular, the rate is $\lambda^{\frac{1}{2}}\abs \xi^{\frac{1}{2}}$ when $\abs \xi \geq\lambda$, hence for every non-zero mode if $\lambda\leq1$. The frequency power $\frac12$ is optimal.

In the autonomous case $A=A(v)$, there exist a probability measure $\mu$ on $\mathbb S^{n-1}$ and a strongly continuous, positivity-preserving solution semigroup for~\eqref{e:spherical-equation} on $L^2(\pi)$, where $\pi=\dd x\otimes\mu$ and $\dd x$ is normalised Haar measure on $\T^n$. The measure $\mu$ is the unique invariant velocity probability for this semigroup. In particular,
\begin{equation}\label{e:spherical-invariant}
\int_{\mathbb S^{n-1}}A:\nabla_S^2\varphi\dd\mu=0
\qquad\forall\varphi\in C^\infty(\mathbb S^{n-1}).
\end{equation}
For every initial datum $u(s)\in L^2(\pi)$, the solution $u\in C([s,\infty);L^2(\pi))$ supplied by this semigroup conserves the mean $m=\int u(s)\dd\pi$ and satisfies
\begin{equation}\label{e:spherical-L2-gap}
\norm{u(t)-m}_{L^2(\pi)}
\leq C\exp\bigl(-c(t-s)\min\{\lambda,\lambda^{-1}\}\bigr)\norm{u(s)-m}_{L^2(\pi)},
\quad t\geq s.
\end{equation}
On the invariant subspace $\big\{f\in L^2(\pi):\int f\dd\pi=0\big\}$, the generator of this solution semigroup has spectrum contained in
\[
\bigl\{z\in\C:\operatorname{Re}z
\leq-c\min\{\lambda,\lambda^{-1}\}\bigr\}.
\]
\end{theorem}
For constant spherical diffusion, Coti Zelati, Dietert and G\'erard-Varet~\cite[Proposition~1.6]{CotiZelatiDietertGerardVaret2023} obtained the square-root enhanced dissipation scale in $L^2$; hypocoercivity for smooth spherical velocity models was established in~\cite{GrothausStilgenbauer2014,BaudoinTardif2018}. Here the Fourier-kernel estimate is uniform in the starting velocity and time and requires only measurable, uniformly elliptic diffusion coefficients. The $L^2$ estimate does not require any formula or bounds on the invariant measure.

\subsection{Strategy}\label{ss:strategy}

For the Aleksandrov and interior regularity estimates, we first assume smooth coefficients and then pass to measurable coefficients as explained in Section~\ref{s:remarks}. The constants are independent of coefficient derivatives. Section~\ref{s:mixing} treats the spherical evolution directly for measurable coefficients, using the Markov construction described there.

\subsubsection{The cancellation argument}

The starting point is to decompose the dynamics into spatial waves. Transport along a velocity path shifts a wave by the spatial displacement accumulated along that path. We prescribe two curves around which velocity can evolve, choosing their accumulated displacements to differ by half a wavelength. The corresponding waves then have opposite phases, thus resulting in destructive interference.

Diffusion spreads mass around each curve, so we retain a small amount of the evolution in a tube around it. The Fourier phases in the spatial variable remain nearly opposite, while the two velocity marginals need not agree. The parabolic Harnack inequality provides a common minorant for the velocity marginals: a positive measure bounded above by both marginals. The contributions corresponding to this lower bound cancel up to the small phase error. This reduces the total variation (in velocity) for the given spatial frequency by a fixed factor, as in Döblin's argument~\cite{Doeblin1940}; see Figure~\ref{fig:doeblin-fourier}. The reduction is uniform in the starting velocity and can therefore be repeated. Scaling determines how long each repetition takes at a given frequency, yielding the exponential Fourier decay in~\eqref{eq:DF}.


\subsubsection{Variables for the cancellation argument}

For $A(t,v)$ we apply this argument in the original time variable. For $a(x,v)$ we use position as time on a velocity interval of fixed non-zero sign. In suitable coordinates, both cases reduce to an equation of the form 
\[
    \partial_\sigma+B(\sigma,y):D_y^2+b(y)\cdot\nabla_z
    \qquad\text{on }\R_\sigma\times D_y\times\R^n_z\,,
    \refstepcounter{equation}\tag{\theequation}
    \label{e:intro-Lop}
\]
where the diffusion acts in variables $y$, the transport acts in variables $z$, $D\subset\R^n$, $B$ is uniformly elliptic, and $b$ is smooth and uniformly non-degenerate. Since $B$ does not depend on $z$, each Fourier wave in $z$ evolves independently. The transport rate $b(y)$ changes with velocity, allowing different velocity paths to produce different phases. The case $A(t,v)$ corresponds to $b(y)=y$ on $\R^n$. In the case $a(x,v)$ in dimension one, the derivative of $b$ is bounded above and below by positive constants.

More generally, the same reduction is available for $A=A(\ell(t,x),v)$ with a linear form $\ell(t,x)=\alpha t+\beta\cdot x$, locally on bounded velocity regions where $\alpha+\beta\cdot v$ stays away from zero. Indeed, $(\partial_t+v\cdot\nabla_x)\ell=\alpha+\beta\cdot v$, so $\ell$ is monotone along transport there and can serve as time, with the appropriate orientation. A global estimate also requires combining these local bounds.


\subsubsection{The kinetic Aleksandrov estimate}

To prove Theorem~\ref{thm:aleksandrov-sharp} for $A=A(t,v)$, we combine Fourier decay with the parabolic Aleksandrov estimate. Integrating the adjoint kernel in position removes transport and gives its velocity marginal $P$. The parabolic estimate controls this marginal in $L^{\widetilde q}$ in time and velocity for $1\leq\widetilde q\leq1+\frac{1}{n}$. On each interval $(T,2T)$ of time elapsed since the start, Fourier decay gives a corresponding bound for the spatial supremum of $K$. Interpolation and summation over these intervals give an $L^q$ bound for the full space-time kernel when $1<q<1+\frac{1}{2n}$. Duhamel's formula and H\"older's inequality then give the source range $p>2n+1$.

For $a(x,v)$ in dimension one, the new transported variable involves the old time, on which the coefficient does not depend. Position can serve as time as long as velocity stays away from zero. Near zero we need narrower velocity intervals. A Harnack comparison controls how their returning mass decreases with elapsed time, allowing us to combine the interval estimates for every $p>4$; see Section~\ref{s:spatial-roughness}.

\subsubsection{Spatial regularity and Hölder regularity}

Exponential Fourier decay gives spatial smoothness by Fourier inversion. To obtain interior regularity, we localise the cancellation argument and bound the spatial derivatives of a solution. We then fix the position and treat $-v\cdot\nabla_xu$ as a bounded source in the parabolic equation in $(t,v)$. The parabolic Krylov--Safonov theory gives Hölder continuity in time and velocity, which combines with the spatial derivative bound to give Hölder continuity in all variables.

\subsubsection{Mixing estimates for rough spherical velocity diffusion}

On the sphere, we adapt the velocity paths in local charts. The same Harnack comparison and cancellation yield square-root frequency decay. Non-zero torus frequencies are bounded away from zero, while compactness of the sphere and parabolic Harnack give mixing of the zero mode. Together these give mixing in total variation. Interpolation with the bound supplied by positivity gives weighted $L^2$ decay. For autonomous coefficients, taking the adjoint and reflecting position transfers this decay to the solution semigroup and gives the spectral gap.

\textbf{Acknowledgements.}
The authors thank Luis Silvestre for sharing his knowledge and intuition about this problem through many discussions over the years. The first author acknowledges support of the Fondation Sciences Mathématiques de Paris. The second author was supported by NSF Grant DMS-2143668.

\textbf{Disclosure of AI assistance.}
The first author conducted an extensive interactive proof search with GPT-5.6 Sol through Codex, providing the authors’ notes as well as proposals and source material, and making numerous choices and interventions. During this process, the model generated the first probabilistic calculation underlying the Döblin--Fourier estimate and helped write a preliminary draft.
The authors checked, selected, rewrote, simplified, unified and interpreted the arguments presented here and take full responsibility for the manuscript. After this proof-development stage, the present short paper was written by the authors, with Codex used for proofreading and editorial revision. The figures were created with Claude-Opus-5.5. The companion manuscript~\cite{Companion}, which was drafted in part with Codex and is not intended for publication, contains the detailed Aleksandrov arguments for both coefficient classes, the passage to H\"older continuity, and the further Bellman improvement for autonomous coefficients in dimension one. It also records the stationary counterexample construction. 

\section{Döblin--Fourier cancellation}
\label{s.fourier}

We show that spatial waves decay by following two families of velocity paths. Transport along the two families produces waves with nearly opposite phases. To obtain cancellation, we need a common positive lower bound for their velocity marginals. The parabolic Harnack inequality gives this bound after both families evolve briefly in a fixed velocity ball.

We prove the estimate for smooth $A=A(t,v)$, with constants independent of its derivatives. Reflecting time and position turns the original equation into a terminal-value problem for
\[
    \mathcal K=\partial_s+v\cdot\nabla_x+\mbox{Tr}(A(s,v)D_v^2).
\]
Here $A(s,v)$ denotes the reflected coefficient $A(-s,v)$. We denote by $K_{s,t}(x,v,w)$ the adjoint fundamental solution of $\mathcal K$. It is non-negative and satisfies
\[
    \partial_t K_{s,t}+w\cdot\nabla_xK_{s,t}
    =\sum_{i,j}\partial_{w_iw_j}\bigl(A_{ij}(t,w)K_{s,t}\bigr),
    \qquad K_{s,s}=\delta_0(x)\delta_v(w).
\]
Thus $s,v$ are the starting time and velocity, $t,w$ the terminal time and velocity, and $x$ the position displacement. All kernels below have their two time indices in chronological order. For terminal data $g$ at time $t>s$, the solution of $\mathcal K U=0$ at $(s,x_0,v)$ is
\[
    \int_{\R^{2n}}g(x_0+x,w)K_{s,t}(x,v,w)\dd x\dd w.
\]
The kernel has mass one. Its evolution can be split at any intermediate time by composing the kernels for the two successive intervals. Integrating $K_{s,t}$ in time over $s<t<S$ gives the space-time Green measure. Integrating a source $F(t,x_0+x,w)$ against this measure gives the solution of $\mathcal K U=-F$ with zero terminal data at $S$, by Duhamel's formula.

We denote the velocity marginal by $P$:
\[
    P_{s,t}(v,w):=\int_{\R^n}K_{s,t}(x,v,w)\dd x.
\]
It is the parabolic adjoint fundamental solution in velocity. For a velocity domain $D$, the superscript $D$ on either kernel means that mass is discarded at the velocity boundary. Thus $K^D$ is the kinetic Dirichlet kernel, and $P^D$ is its velocity marginal; both have mass at most one and are bounded above by their full-space counterparts. We write $P_{s,t}(v,\dd w)=P_{s,t}(v,w)\dd w$, and likewise for $P^D$, when using these kernels as measures.

Given a Borel set $E\subset\R^n$, the Fourier transform is the complex measure
\[
    \widehat K_{s,t}^\xi(v,E)
    :=\int_{E}\int_{\R^n}e^{-i\xi\cdot x}K_{s,t}(x,v,w)\dd x\dd w .
    \refstepcounter{equation}\tag{\theequation}
    \label{e:fourier-kernel}
\]
At frequency zero,~\eqref{e:fourier-kernel} gives the velocity marginal $P_{s,t}$. At other frequencies, the exponential still has modulus one. Thus the variation measure of $\widehat K_{s,t}^\xi$ is bounded by $P_{s,t}$, and $\norm{\widehat K_{s,t}^\xi(v,\cdot)}_{\TV}\leq1$. We write $\widehat K^{D,\xi}_{s,t}$ for the Fourier transform of $K^D_{s,t}$. The cancellation will improve the mass bound at non-zero frequencies.

\begin{proposition}[Döblin--Fourier cancellation estimate]
  \label{p:decay}
  There are $C,c>0$, depending only on $n,\frac{\Lambda}{\lambda}$, such that, for $0\leq s<t\leq1$, $v,\xi\in\R^n$,
  \begin{equation}
    \label{e:decay-W}
    \norm{\widehat K_{s,t}^\xi(v,\cdot)}_{\TV}
    \leq C\exp\bigl(-c\,(t-s) \lambda^{\frac{1}{3}}\abs\xi^{\frac{2}{3}}\bigr).
  \end{equation}
\end{proposition}

We begin with a wave of unit frequency. The key is to find a fixed time after which its total variation has decreased by a fixed factor, uniformly over the starting velocity. Once this is known, scaling gives the corresponding time for any non-zero frequency. Repeating the decrease on successive time intervals then gives exponential decay.
\begin{proposition}
  \label{p:block}
  There are $L>0$ and $\delta\in(0,1)$, depending only on $n,\lambda,\Lambda$, such that, whenever the kernel is defined on $[s,s+L]$,
  \begin{equation}
    \label{e:block}
    \sup_{v\in\R^n,\ \abs e=1}
    \norm{\widehat K_{s,s+L}^e(v,\cdot)}_{\TV}\leq1-\delta .
  \end{equation}
\end{proposition}

We first derive the decay estimate from a contraction at unit frequency, then construct that contraction.

\subsection{Scaling to unit frequency}

To deduce Proposition~\ref{p:decay}, we first reduce to lower ellipticity bound one. Under $x=\sqrt\lambda X$ and $v=\sqrt\lambda V$, the matrix becomes $\widetilde A(t,V)=\lambda^{-1}A(t,\sqrt\lambda V)$, with ellipticity constants $1$ and $\frac{\Lambda}{\lambda}$, while the Fourier frequency becomes $\widetilde\xi=\sqrt\lambda\xi$. Rescaling the terminal velocity preserves total variation. Hence a decay rate $\abs{\widetilde\xi}^{\frac{2}{3}}$ gives the rate $\lambda^{\frac{1}{3}}\abs\xi^{\frac{2}{3}}$ in the original variables, with constants depending only on $n,\frac{\Lambda}{\lambda}$. In this and the next subsection we use the normalised variables and omit tildes.

At frequency $\xi$, a spatial displacement of order $\abs\xi^{-1}$ changes the phase by order one. Diffusion changes velocity by order $h^{\frac{1}{2}}$ over a time interval of length $h$, so transport produces a relative displacement of order $h^{\frac{3}{2}}$. The corresponding time scale is therefore $h=\abs\xi^{-\frac{2}{3}}$.

To make this precise, we fix $s_0,v_0$ and $\xi\ne0$, and set $r=\abs\xi^{-\frac{1}{3}}$ and $e=r^3\xi$. Under
\[
  t=s_0+r^2\tau,\qquad v=v_0+rV,\qquad
  x=r^2\tau v_0+r^3X,
\]
$\mathcal K$ becomes $r^{-2}\bigl(\partial_\tau+V\cdot\nabla_X+\mbox{Tr}(A_rD_V^2)\bigr)$, where
$A_r(\tau,V)=A(s_0+r^2\tau,v_0+rV)$ has the same ellipticity constants.  The shift in $x$ only multiplies the Fourier kernel by a number of modulus one, and rescaling the terminal velocity preserves total variation. Moreover, $r^3\xi=e$ and $\abs e=1$.  Thus~\eqref{e:block} gives, on every admissible interval,
\begin{equation}
  \label{e:scaled-block}
  \sup_v\norm{\widehat K_{s,s+L\abs\xi^{-\frac{2}{3}}}^\xi(v,\cdot)}_{\TV}
  \leq1-\delta .
\end{equation}

\subsection{Iterating the contraction}

For $s<r<t$, composition and~\eqref{e:fourier-kernel} give
\[
  \widehat K_{s,t}^\xi(v,E)
  =\int_{\R^n}\widehat K_{r,t}^\xi(v',E)\widehat K_{s,r}^\xi(v,\dd v'),
\]
hence
\[
  \norm{\widehat K_{s,t}^\xi(v,\cdot)}_{\TV}
  \leq\norm{\widehat K_{s,r}^\xi(v,\cdot)}_{\TV}
  \sup_{v'}\norm{\widehat K_{r,t}^\xi(v',\cdot)}_{\TV}.
\]
Iterating~\eqref{e:scaled-block} over
$N=\lfloor\frac{(t-s)\abs\xi^{\frac{2}{3}}}{L}\rfloor$ intervals and using the bound $1$ on the remainder yields
\begin{equation}
  \label{e:iteration}
  \norm{\widehat K_{s,t}^\xi(v,\cdot)}_{\TV}
  \leq(1-\delta)^N    \leq(1-\delta)^{-1} e^{-\frac{-\log(1-\delta)}{L}(t-s)\abs\xi^{\frac{2}{3}}}.
\end{equation}
For $\xi=0$ the left-hand side is at most one. The constants $L$ and $\delta$ depend only on $n,\frac{\Lambda}{\lambda}$ in the normalised variables; the number $N$ also depends on elapsed time and frequency. Undoing the ellipticity normalisation restores the factor $\lambda^{\frac{1}{3}}$. Thus Proposition~\ref{p:block} implies Proposition~\ref{p:decay}, and it remains to prove Proposition~\ref{p:block}.

\subsection{Mass in a moving tube}

We first show that a positive amount of mass can follow a prescribed velocity path. We allow velocities to stay within a narrow tube around the path, and discard mass when it leaves the tube. The width of the tube leaves room for diffusion, while keeping the accumulated spatial displacement close to the integral of the prescribed path. A barrier gives a lower bound for the mass that remains.

We take $\gamma:[s,s+T]\to\R^n$ piecewise $C^1$, with $\abs{\gamma'}\leq H$, and denote by $K^{\gamma,\eta}$ the Dirichlet fundamental solution with zero boundary values on $\abs{w-\gamma(t)}=\eta$.  For an integer $N=N(n,\lambda,\Lambda)$ large enough,
\[
    \psi_t(w)=\left(1-\frac{\abs{w-\gamma(t)}^2}{\eta^2}\right)_+^N
\]
satisfies, on each $C^1$ piece of $\gamma$,
\[
    \bigl(\partial_t+\mbox{Tr}(A(t,w)D_w^2)\bigr)\psi_t
    \geq-C(\eta^{-2}+H^2)\psi_t.
\]
The maximum principle applied to the corresponding exponential multiple of $\psi$ gives
\[
    \int_{\R^{2n}}K^{\gamma,\eta}_{s,s+T}(x,\gamma(s),w)\dd x\dd w
    \geq e^{-C(\eta^{-2}+H^2)T}.
    \refstepcounter{equation}\tag{\theequation}
    \label{e:path-mass}
\]
Moreover, for every unit vector $e$, the support of $K^{\gamma,\eta}_{s,s+T}(\cdot,\gamma(s),\cdot)$ is contained in
\[
    \left\{(x,w):
    \left|e\cdot x-\int_s^{s+T}e\cdot\gamma(t)\dd t\right|\leq\eta T,
    \ \abs{w-\gamma(s+T)}\leq\eta\right\}.
    \refstepcounter{equation}\tag{\theequation}
    \label{e:path-phase}
\]

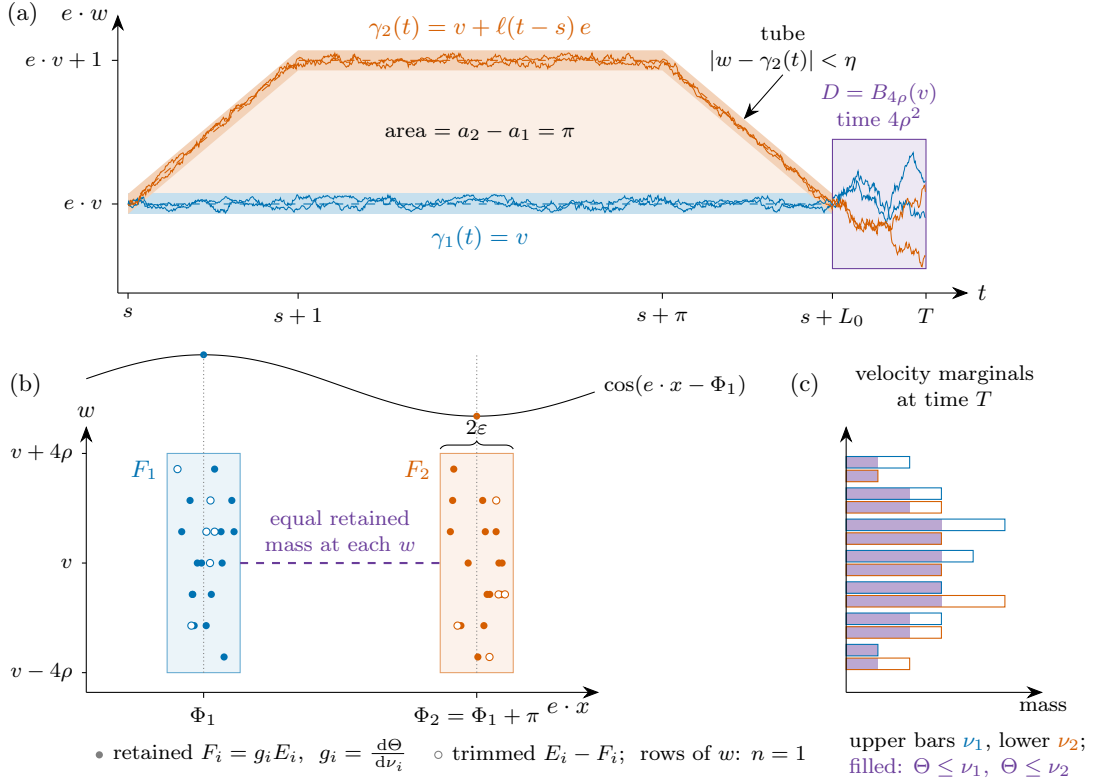
\begin{figure}[t]
  \centering
\begin{tikzpicture}[
    >={Stealth[length=2.2mm]},
    lab/.style={font=\small},
    note/.style={font=\footnotesize},
  ]
  \def\tw{0.07}
  \def\Lz{4.1416}
  \def\Cz{0.55}
  \def\rr{0.45}
  \begin{scope}[x=2.25cm,y=1.9cm]
    \fill[dftwo!10] (0,0) -- (1,1) -- (3.1416,1) -- (\Lz,0) -- cycle;
    \fill[dfone!22] (0,-\tw) rectangle (\Lz,\tw);
    \fill[dftwo!28] (0,\tw) -- (1,1+\tw) -- (3.1416,1+\tw) -- (\Lz,\tw)
                -- (\Lz,-\tw) -- (3.1416,1-\tw) -- (1,1-\tw) -- (0,-\tw) -- cycle;
    \filldraw[fill=dftheta!12,draw=dftheta,thin] (\Lz,-\rr) rectangle (\Lz+\Cz,\rr);
    \draw[dfone,thin,dashed] (0,0) -- (\Lz,0);
    \draw[dftwo,thin,dashed] (0,0) -- (1,1) -- (3.1416,1) -- (\Lz,0);
    \foreach \j in {0,1}{
      \draw[dfone,line width=0.35pt] plot file {\figdata f1_path1_\j.dat};
      \draw[dftwo,line width=0.35pt] plot file {\figdata f1_path2_\j.dat};
    }
    \draw[->] (-0.06,-0.62) -- (\Lz+\Cz+0.25,-0.62) node[right,lab] {$t$};
    \draw[->] (-0.06,-0.62) -- (-0.06,1.32) node[left,lab] {$e\cdot w$};
    \foreach \x/\l in {0/s, 1/s+1, 3.1416/s+\pi, \Lz/s+L_0, \Lz+\Cz/T}
      \draw (\x,-0.62) -- ++(0,-0.04) node[below,note] {$\l$};
    \foreach \y/\l in {0/e\cdot v, 1/e\cdot v+1}
      \draw (-0.06,\y) -- ++(-0.035,0) node[left,note] {$\l$};
    \node[lab,dftwo,above] at (2.07,1.08) {$\gamma_2(t)=v+\ell(t-s)\,e$};
    \node[lab,dfone,below] at (2.07,-0.09) {$\gamma_1(t)=v$};
    \node[note,align=center] at (2.07,0.5) {area $=a_2-a_1=\pi$};
    \draw[note,<-] (3.62,0.52+\tw) -- (3.85,0.9)
      node[above,align=center,inner sep=1pt] {tube\\ $\abs{w-\gamma_2(t)}<\eta$};
    \node[note,dftheta,align=center,above] at (\Lz+\Cz/2,\rr)
      {$D=B_{4\rho}(v)$\\ time $4\rho^2$};
    \node[lab] at (-0.62,1.32) {(a)};
  \end{scope}

  \begin{scope}[shift={(1.0cm,-4.75cm)},x=1.15cm,y=1.45cm]
    \def\ep{0.42}
    \def\Pa{0}
    \def\Pb{3.1416}
    \draw[thin,domain=-1.35:4.5,samples=90,smooth] plot (\x,{1.62+0.28*cos(\x r)});
    \node[note,right] at (4.5,1.62) {$\cos(e\cdot x-\Phi_1)$};
    \fill[dfone] (\Pa,1.9) circle (1.3pt);
    \fill[dftwo] (\Pb,1.34) circle (1.3pt);
    \filldraw[fill=dfone!8,draw=dfone!60,thin] (\Pa-\ep,-1) rectangle (\Pa+\ep,1);
    \filldraw[fill=dftwo!8,draw=dftwo!60,thin] (\Pb-\ep,-1) rectangle (\Pb+\ep,1);
    \foreach \x in {\Pa,\Pb} \draw[dfgrey,densely dotted] (\x,-1) -- (\x,1.9);
    \pgfmathsetseed{11}
    \foreach \w/\th/\ea/\eb in {-0.857/1/0/1, -0.571/2/1/1, -0.286/3/0/2, 0/3/1/0,
                                 0.286/3/2/0, 0.571/2/1/1, 0.857/1/1/0}{
      \foreach \P/\c/\ex in {\Pa/dfone/\ea, \Pb/dftwo/\eb}{
        \foreach \k in {1,...,\th}{
          \pgfmathsetmacro\dx{0.85*\ep*rand}
          \fill[\c] (\P+\dx,\w) circle (1.35pt);
        }
        \ifnum\ex>0
          \foreach \k in {1,...,\ex}{
            \pgfmathsetmacro\dx{0.85*\ep*rand}
            \draw[\c,thin,fill=white] (\P+\dx,\w) circle (1.35pt);
          }
        \fi
      }
    }
    \draw[dftheta,thick,dashed] (\Pa+\ep,0) -- (\Pb-\ep,0);
    \node[note,dftheta,align=center,above] at (1.5708,0.02)
      {equal retained\\ mass at each $w$};
    \draw[->] (-1.35,-1.18) -- (4.55,-1.18) node[below left,lab] {$e\cdot x$};
    \draw (\Pa,-1.18) -- ++(0,-0.05) node[below,note] {$\Phi_1$};
    \draw (\Pb,-1.18) -- ++(0,-0.05) node[below,note] {$\Phi_2=\Phi_1+\pi$};
    \draw[->] (-1.35,-1.18) -- (-1.35,1.25) node[above,lab] {$w$};
    \foreach \y/\l in {-1/v-4\rho, 0/v, 1/v+4\rho}
      \draw (-1.35,\y) -- ++(-0.05,0) node[left,note] {$\l$};
    \draw[decorate,decoration={brace,amplitude=3pt},thin]
      (\Pb-\ep,1.04) -- (\Pb+\ep,1.04) node[midway,above=2pt,note] {$2\varepsilon$};
    \node[lab,dfone,left] at (\Pa-\ep,0.85) {$F_1$};
    \node[lab,dftwo,left] at (\Pb-\ep,0.85) {$F_2$};
    \begin{scope}[shift={(-1.2,-1.75)}]
      \fill[dfgrey] (0,0) circle (1.35pt);
      \node[note,right] at (0.05,0) {retained $F_i=g_iE_i$, \ $g_i=\frac{\dd\Theta}{\dd\nu_i}$};
      \draw[dfgrey,fill=white] (3.9,0) circle (1.35pt);
      \node[note,right] at (3.95,0) {trimmed $E_i-F_i$;\ \ rows of $w$: $n=1$};
    \end{scope}
    \node[lab] at (-2.05,1.62) {(b)};
  \end{scope}

  \begin{scope}[shift={(9.5cm,-4.75cm)},x=0.42cm,y=1.45cm]
    \foreach \w/\th/\ea/\eb in {-0.857/1/0/1, -0.571/2/1/1, -0.286/3/0/2, 0/3/1/0,
                                 0.286/3/2/0, 0.571/2/1/1, 0.857/1/1/0}{
      \fill[dftheta!45] (0,\w+0.01) rectangle (\th,\w+0.115);
      \fill[dftheta!45] (0,\w-0.115) rectangle (\th,\w-0.01);
      \draw[dfone,thin] (0,\w+0.01) rectangle (\th+\ea,\w+0.115);
      \draw[dftwo,thin] (0,\w-0.115) rectangle (\th+\eb,\w-0.01);
    }
    \draw[->] (0,-1.18) -- (0,1.25);
    \draw[->] (0,-1.18) -- (6.2,-1.18) node[below,note] {mass};
    \node[note,align=left,right] at (-0.2,-1.75) {upper bars \textcolor{dfone}{$\nu_1$},
      lower \textcolor{dftwo}{$\nu_2$};\\ \textcolor{dftheta}{filled: $\Theta\le\nu_1,\ \Theta\le\nu_2$}};
    \node[note,align=center] at (3.1,1.62) {velocity marginals\\ at time $T$};
    \node[lab] at (-1.3,1.62) {(c)};
  \end{scope}
\end{tikzpicture}
  \caption{\small{Döblin--Fourier cancellation at unit frequency $e$. \textbf{(a)} The paths $\gamma_1,\gamma_2$ with the tubes in which mass is retained, then the common continuation in $D=B_{4\rho}(v)$. Since $e\cdot x=\int e\cdot w\dd t$, the central displacements differ by the shaded area, $a_2-a_1=\pi$. \textbf{(b)} At time $T$ the retained mass lies in the windows $\abs{e\cdot x-\Phi_i}\leq\varepsilon$, at a crest and a trough of $\cos(e\cdot x-\Phi_1)$; rows of terminal velocity $w$ are drawn for $n=1$. Trimming $E_i$ (all dots) to $F_i=g_iE_i$ (filled dots) leaves equal mass $\Theta(\dd w)$ of both families at each $w$. The central paths give exactly opposite phase factors $e^{-i\Phi_1}$ and $e^{-i\Phi_2}=-e^{-i\Phi_1}$; the phase factors of the retained mass differ from these central factors $e^{-i\Phi_i}$ by at most $\varepsilon$. \textbf{(c)} The velocity marginals $\nu_i\geq\Theta$. Widths are not to scale.}}
  \label{fig:doeblin-fourier}
\end{figure}

\subsection{Opposite phases}

One prescribed curve stays at the initial velocity; the other makes an excursion in the direction of the frequency and returns. These curves specify the centres of the two tubes; see Figure~\ref{fig:doeblin-fourier}(a). We choose the excursion so that the two accumulated displacements differ by $\pi$ in that direction.

We fix $s,v$ and a unit vector $e$, and put $L_0=1+\pi$ and
\[
    \ell(t)=\min\{t,1,L_0-t\},\qquad
    \gamma_1(s+t)=v,\qquad \gamma_2(s+t)=v+\ell(t)e.
\]
Then $\gamma_i(s)=\gamma_i(s+L_0)=v$, $\abs{\gamma_i'}\leq1$, and
\[
    a_i:=\int_s^{s+L_0}e\cdot\gamma_i(t)\dd t,
    \qquad a_2-a_1=\int_0^{L_0}\ell(t)\dd t=\pi.
    \refstepcounter{equation}\tag{\theequation}
    \label{e:opposite-paths}
\]
We choose fixed $0<\rho<128^{-\frac{1}{3}}$ and
$0<\eta<\min\{\rho,(8L_0)^{-1}\}$.  The Dirichlet fundamental solutions along $\gamma_i$ give two positive measures $R_i$:
\[
    R_i(\dd x\dd w)=K^{\gamma_i,\eta}_{s,s+L_0}(x,v,w)\dd x\dd w.
\]
By~\eqref{e:path-mass}, their masses $m_i$ satisfy
\[
    m_i\geq m_0:=e^{-C(\eta^{-2}+1)L_0}>0.
\]
They are dominated by the full kernel at time $s+L_0$; consequently
\[
    \widetilde R_i:=\frac{m_0}{2m_i}R_i,
    \qquad \widetilde R_1+\widetilde R_2
    \leq K_{s,s+L_0}(x,v,w)\dd x\dd w.
    \refstepcounter{equation}\tag{\theequation}
    \label{e:two-parts}
\]
The factor $\frac{1}{2}$ prevents double counting even if the two measures overlap. Each $\widetilde R_i$ has mass $\frac{m_0}{2}$; by~\eqref{e:path-phase}, its support satisfies $\abs{w-v}\leq\eta$ and $\abs{e\cdot x-a_i}\leq\eta L_0$.

\subsection{A common minorant}

The two tubes end in a common velocity ball. The distributions within that ball may differ, so we let both measures evolve a little longer there. The parabolic Harnack inequality then gives a common minorant for their velocity marginals. The extra time is short enough that the phases remain nearly opposite.

We set $D=B_{4\rho}(v)$ and continue both $\widetilde R_i$ up to $T=s+L_0+4\rho^2$ using $K^D$. We denote the resulting measures by $E_i$. Positivity and composition give
\[
    E_1+E_2\leq M,
    \qquad M(\dd x\dd w):=K_{s,T}(x,v,w)\dd x\dd w.
    \refstepcounter{equation}\tag{\theequation}
    \label{e:continued-parts}
\]
For each non-negative terminal test function $\varphi$, the function
\[
    (\tau,v')\longmapsto\int\varphi(w)P^D_{T-\tau,T}(v',\dd w)
\]
solves the forward non-divergence parabolic equation with diffusion matrix $A(T-\tau,v')$ in $D$. We apply the parabolic Harnack inequality to this non-negative function, comparing $\tau=\rho^2$ with $\tau=4\rho^2$. Together with~\eqref{e:path-mass} for the constant path, this gives $h,q_0>0$, depending only on $n,\lambda,\Lambda$, such that
\[
    \begin{aligned}
    &P^D_{T-4\rho^2,T}(v',\cdot)
    \geq hP^D_{T-\rho^2,T}(v,\cdot), \qquad(\abs{v'-v}<\rho),\\
    &P^D_{T-\rho^2,T}(v,B_{4\rho}(v))\geq q_0.
    \end{aligned}
    \refstepcounter{equation}\tag{\theequation}
    \label{e:parabolic-minoration}
\]
By~\eqref{e:parabolic-minoration}, the velocity marginal $\nu_i$ of $E_i$ satisfies
\[
    \nu_i\geq\Theta:=\frac{hm_0}{2}P^D_{T-\rho^2,T}(v,\cdot),
    \qquad \Theta(\R^n)\geq\theta_0:=\frac{hm_0q_0}{2}>0.
    \refstepcounter{equation}\tag{\theequation}
    \label{e:common-average}
\]
We use only the portion of each measure supplied by the common lower bound $\Theta$. Since $\Theta\leq\nu_i$, the density $g_i=\frac{\dd\Theta}{\dd\nu_i}$ lies between zero and one. Thus $F_i(\dd x\dd w):=g_i(w)E_i(\dd x\dd w)$ satisfies $F_i\leq E_i$, and its velocity marginal is $\Theta$. In particular, $F_1+F_2\leq M$. Finally, we put
\[
    \Phi_i=a_i+4\rho^2e\cdot v,
    \qquad \varepsilon=\eta L_0+16\rho^3<\frac{1}{4}.
\]
On the last interval, of length $4\rho^2$, the Dirichlet fundamental solution is
confined to $B_{4\rho}(v)$.  The additional spatial displacement in
direction $e$ therefore lies within $4\rho\,4\rho^2=16\rho^3$ of
$4\rho^2e\cdot v$.
Together with~\eqref{e:path-phase}, this gives
\[
    \abs{e\cdot x-\Phi_i}\leq\varepsilon
    \quad\text{on the support of }F_i,
    \qquad \Phi_2-\Phi_1=\pi.
    \refstepcounter{equation}\tag{\theequation}
    \label{e:phase-windows}
\]

\subsection{The cancellation}

The portions $F_1$ and $F_2$ supplied by the common minorant have Fourier phases close to $e^{-i\Phi_1}$ and $e^{-i\Phi_2}$. Replacing the phases by these fixed values gives the opposite measures $e^{-i\Phi_1}\Theta$ and $e^{-i\Phi_2}\Theta$, whose sum is zero. We bound the phase error and the contribution of the mass outside $F_1+F_2$.

For a positive measure $F$ on $\R_x^n\times\R_w^n$, we write
$\widehat F(E)=\int_{\R^n\times E}e^{-ie\cdot x}F(\dd x\dd w)$.  Since the velocity marginal of $F_i$ is $\Theta$,~\eqref{e:phase-windows} gives
\[
    \norm{\widehat F_i-e^{-i\Phi_i}\Theta}_{\TV}
    \leq\varepsilon\Theta(\R^n).
\]
Here we used $\abs{e^{-ia}-e^{-ib}}\leq\abs{a-b}$.  Since
$F_1+F_2\leq M$, we may now write
\[
    \widehat M
    =(\widehat F_1-e^{-i\Phi_1}\Theta)
      +(\widehat F_2-e^{-i\Phi_2}\Theta) +(e^{-i\Phi_1}+e^{-i\Phi_2})\Theta
      +\widehat{(M-F_1-F_2)}.
\]
The first two terms are the phase errors just estimated. The third term vanishes because $\Phi_2-\Phi_1=\pi$. The last term comes from a positive measure of mass $1-2\Theta(\R^n)$, so $\norm{\widehat{(M-F_1-F_2)}}_{\TV}\leq1-2\Theta(\R^n)$. Therefore
\[
    \norm{\widehat K_{s,T}^e(v,\cdot)}_{\TV}
    =\norm{\widehat M}_{\TV}\leq
      2\varepsilon\Theta(\R^n)+1-2\Theta(\R^n)\leq1-2(1-\varepsilon)\theta_0=:1-\delta<1.
\]
This is~\eqref{e:block} with $L=L_0+4\rho^2$, proving Proposition~\ref{p:block}. Proposition~\ref{p:decay} then follows from~\eqref{e:iteration}.

\section{Kinetic Aleksandrov maximum principle}
\label{sec:amp}

We now prove the Aleksandrov maximum principle by estimating the solution produced by a source. This solution is obtained by integrating the source against the kernel, so we need an integrability bound for that kernel. Fourier decay gives control of its supremum in position. The parabolic Aleksandrov estimate gives control of its integral in position, as a function of time and velocity. Interpolating between these two bounds gives the required integrability in all variables. Since the parabolic estimate already involves a time integral, we combine the bounds on intervals $(T,2T)$ of time elapsed since the start. The restriction on the exponent comes from summing their contributions as $T$ tends to zero.

We continue with smooth $A=A(t,v)$ and work at radius $R=1$. The final change of scale and the approximation in Section~\ref{s:remarks} give the stated estimate.

\subsection{The velocity marginal}

Integrating the kernel in position removes the transport term and leaves the parabolic kernel $P$. We can therefore apply the classical parabolic Aleksandrov estimate. We fix a starting time and velocity $(s,v)$, and consider a time interval of length $0<S\leq1$, with $s+S\leq1$. The variable $\tau$ denotes the time elapsed since $s$. The marginal $P_{s,s+\tau}(v,w)$ equals the $L^1_x$ norm of $K_{s,s+\tau}(\cdot,v,w)$. As a function of $(\tau,w)$, it solves the adjoint parabolic equation. Integrating terminal data against it gives the solution of the backward non-divergence equation in the starting variables. The parabolic Aleksandrov estimate~\cite{Krylov1976,Tso1985}, followed by duality and interpolation with the mass bound, gives
\[
    \norm{P_{s,s+\tau}(v,w)}_
    {L^\gamma((0,S)\times\R_w^n)}
    \leq CS^{\beta_\gamma},
    \qquad
    \beta_\gamma:=\frac{n+2}{2\gamma}-\frac{n}{2},
    \quad1\leq\gamma\leq\frac{n+1}{n}.
    \refstepcounter{equation}\tag{\theequation}
    \label{e:zero-mode}
\]
To obtain this bound over the whole velocity space, we first apply the parabolic estimate in a large ball and control the contribution from its boundary. We take a non-negative smooth compactly supported source $F$ on $(s,s+S)\times\R^n$, and denote by $h$ the bounded solution of
\[
    \partial_rh+\mbox{Tr}(A(r,w)D_w^2h)=-F(r,w)
    \quad\hbox{in }(s,s+S)\times\R^n,
    \qquad h(s+S,\cdot)=0.
\]
With $M$ denoting the supremum of $h$ over $(s,s+S)\times\R^n$, the quadratic barrier in a ball of radius $R$ about any starting velocity bounds the contribution of the lateral boundary values by $\frac{2n\Lambda S M}{R^2}$. The Aleksandrov estimate on the bounded cylinder therefore gives
\[
    M\leq\frac{2n\Lambda S}{R^2}M
       +CR^{\frac{n}{n+1}}\norm{F}_{L^{n+1}}.
\]
We choose $R^2=4n\Lambda S$, so that the first term is $\frac{M}{2}$ and can be moved to the left-hand side. Duality then gives~\eqref{e:zero-mode} at $\gamma=\frac{n+1}{n}$. The time integral of the kernel has total mass $S$, which supplies the $L^1$ bound. Interpolating between these two bounds gives the remaining exponents.

\subsection{A bound in position}

We use the first part of the evolution to make the spatial waves decay, and the second part to apply the parabolic estimate. For $0<T\leq \frac{S}{2}$, we split the evolution at time $s+\frac{T}{2}$. For $T<\tau<2T$, composing the two kernels and using~\eqref{e:zero-mode} gives a density $\widehat K_{s,s+\tau}^\xi(v,w)$ satisfying
\[
    \norm{\widehat K_{s,s+\tau}^\xi(v,w)}_
    {L^\gamma((T,2T)\times\R_w^n)}
    \leq CT^{\beta_\gamma}e^{-cT\abs\xi^{\frac{2}{3}}},
    \qquad1\leq\gamma\leq\frac{n+1}{n}.
    \refstepcounter{equation}\tag{\theequation}
    \label{e:slab-mode}
\]
Indeed, Proposition~\ref{p:decay} bounds the total variation of the first kernel by $Ce^{-cT\abs\xi^{\frac{2}{3}}}$. The variation of the second kernel is bounded by its positive velocity marginal, to which~\eqref{e:zero-mode} applies for every intermediate velocity. Integrating this parabolic bound against the variation of the first kernel gives~\eqref{e:slab-mode}. Thus Fourier cancellation supplies the decay in $\xi$, while the positive marginal supplies the integrability in time and velocity.
Since
\[
  \int_{\R^n}e^{-cT\abs\xi^{\frac{2}{3}}}\dd\xi=CT^{-\frac{3n}{2}},
\]
Applying Fourier inversion and Minkowski's inequality to~\eqref{e:slab-mode} gives a bound for the supremum in position:
\[
  \norm{K_{s,s+\tau}(\cdot,v,\cdot)}_
  {L^\gamma((T,2T)\times\R_w^n;L^\infty(\R_x^n))}
  \leq CT^{\beta_\gamma-\frac{3n}{2}}.
    \refstepcounter{equation}\tag{\theequation}
    \label{e:Fourier-height}
\]

The factor $T^{-\frac{3n}{2}}$ is the inverse volume of a spatial ball of radius $T^{\frac{3}{2}}$, the distance over which diffusion and transport spread mass in time $T$.

\subsection{Spatial interpolation and time integration}

The preceding bounds control two different quantities: the supremum of the kernel in position and its integral in position. We combine them by bounding $q-1$ powers of the kernel by the supremum and integrating the remaining power. We set $\gamma_0:=\frac{n+1}{n}$. If $1<q<\gamma_0$, we choose $q_1$ by $\frac{q-1}{\gamma_0}+\frac{1}{q_1}=1$, so that $1\leq q_1\leq\gamma_0$. Positivity gives, for every $\tau,w$,
\[
    \int_{\R^n}K_{s,s+\tau}(x,v,w)^q\dd x
    \leq
    \left(\sup_xK_{s,s+\tau}(x,v,w)\right)^{q-1}
    P_{s,s+\tau}(v,w).
\]
Hölder's inequality,~\eqref{e:zero-mode} and~\eqref{e:Fourier-height}
therefore give
\[
    \norm{K_{s,s+\tau}(\cdot,v,\cdot)}_
    {L^q((T,2T)\times\R^{2n})}^q
    \leq CT^{1-2n(q-1)}.
    \refstepcounter{equation}\tag{\theequation}
    \label{e:slab-density}
\]
Here we used the identity
$(q-1)(\beta_{\gamma_0}-\frac{3n}{2})+\beta_{q_1}=1-2n(q-1)$.
To sum the elapsed-time intervals, we integrate~\eqref{e:slab-density} against $T^{-1} {\rm d} T$ for $0<T<\frac{S}{2}$ and add the estimate on $(\frac{S}{2},S)$. Each elapsed time below $\frac{S}{2}$ is counted with weight $\log2$. For $q>1$, the integral $\int_0^{\frac{S}{2}}T^{-2n(q-1)}\dd T$ is finite precisely when $q<1+\frac{1}{2n}$. In this range, we obtain
\[
    \norm{K_{s,s+\tau}(\cdot,v,\cdot)}_
    {L^q((0,S)\times\R^{2n})}
    \leq CS^{\frac{1-2n(q-1)}{q}}.
    \refstepcounter{equation}\tag{\theequation}
    \label{e:Green-density}
\]

With $q=\frac{p}{p-1}$, the preceding range for $q$ is equivalent to $p>2n+1$.
After the reflection $(t,x)\mapsto(-t,-x)$, we denote by $F$ the reflected source $f_+$, extended by zero outside the reflected cylinder, and by $s_*$ its terminal time. Integrating this source against the kernel gives
\[
    W(s,x_0,v):=\int_s^{s_*}\int_{\R^{2n}}
    F(t,x_0+x,w)K_{s,t}(x,v,w)\dd x\dd w\dd t.
\]
This is the non-negative solution of $\mathcal K W=-F$ with zero terminal data at $s_*$. Subtracting $W$ from the reflected sub-solution leaves a homogeneous differential inequality. Since $W\geq0$, this subtraction does not increase the boundary values. The maximum principle therefore bounds the difference by the positive boundary supremum. Hölder's inequality and~\eqref{e:Green-density} bound $W$ in terms of the source norm. Returning to the original variables gives, for $P\in Q_1^-(P_0)$,
\[
    u(P)\leq\sup_{\partial_kQ_1^-(P_0)}u_+
    +C\norm{f_+}_{L^p(Q_1^-(P_0))}.
\]
This proves~\eqref{eq:abp-kin} at radius one. On $Q_R^-(P_0)$, rescaling multiplies the source by $R^2$ and the volume by $R^{4n+2}$, giving the source factor $R^{2-\frac{4n+2}{p}}$.

\section{Spatial regularity}
\label{sec:space}

For $A=A(t,v)$, every bounded classical solution of $\mathcal Lu=0$ is $C^\infty$ in the spatial variable $x$ in the interior. The local Fourier decay is integrable after multiplication by any power of $\abs\xi$, so we can differentiate the representation in $x$ arbitrarily many times, without differentiating $A$. To obtain this decay, we account for the velocity boundary: little mass reaches it in a short time, while mass remaining inside for longer undergoes Fourier cancellation. This gives the derivative bounds in~\eqref{e:x-gradient}.

\subsection{Local cancellation}

We set $c_0=\frac{1}{8}$ and work with the Dirichlet kernel $K^D$ for $D=B_R(v_0)$. We consider a solution $U$ of $\mathcal K U=0$ on
\[
    (s,s+c_0R^2)\times\R_x^n\times D,
\]
with prescribed terminal and velocity-boundary data $g$. Its value at $(s,x,v)$ is obtained by integrating these data against a non-negative measure $\omega^D_{s,v}$:
\[
    U(s,x,v)=\int g\bigl(r,x+(r-s)v_0+y,w\bigr)
    \omega^D_{s,v}(\dd r\dd y\dd w).
\]
Here $y=x'-x-(r-s)v_0$ measures the spatial displacement relative to motion at velocity $v_0$. The integration is over $y\in\R^n$ and
\[
    (r,w)\in\bigl(\{s+c_0R^2\}\times D\bigr)
    \cup\bigl((s,s+c_0R^2]\times\partial D\bigr).
\]
The two sets correspond to reaching the terminal time while still inside $D$, or reaching the velocity boundary by that time. Thus the formula includes all the prescribed data. The measure has total mass one, since constant data $g=1$ give the constant solution $U=1$. It is independent of the starting position $x$ because the coefficients are independent of position. We denote its Fourier transform in $y$ by $\widehat\omega_{s,v}^{D,\xi}$. Subtracting the displacement $(r-s)v_0$ only multiplies the Fourier measure by a phase of modulus one, depending on $r$, and hence leaves its total variation unchanged.

This boundary measure can be constructed from $K^D$ itself. For a bounded source $F$, we denote by $W_F^D$ the solution of $\mathcal K W_F^D=-F$ with zero terminal and velocity-boundary data, obtained by integrating $F$ against $K^D$ over later times. If $\phi$ is a smooth extension of the prescribed boundary data, the solution of $\mathcal K U=0$ with those data is $\phi+W_{\mathcal K\phi}^D$. The bound
\[
|W_F^D(r,x,v)|\leq\norm F_\infty
\min\left\{s+c_0R^2-r,\frac{R^2-|v-v_0|^2}{2n\lambda}\right\}
\]
follows by comparison with two barriers: the time remaining until the terminal face, and the quadratic function of velocity in the display. These show that $W_F^D$ tends to zero at the terminal face and the velocity boundary. Comparison and the Riesz representation theorem then give the non-negative measure representing the boundary data.

The Dirichlet kernel $K^D$, started at $(s,x,v)$, is supported where $\abs{x'-x-(r-s)v_0}\leq R(r-s)$. Indeed, for each unit vector $e$, the affine functions $\pm e\cdot[x'-x-(r-s)v_0]-R(r-s)$ vanish at the starting point and have transport derivative $\pm e\cdot(w-v_0)-R\leq0$ in $D$. Composing with a non-negative increasing smooth function that vanishes on $(-\infty,0]$ and applying comparison gives the support bound. To verify that the boundary measure has mass one, we choose a position cutoff equal to one near this support for $s\leq r\leq s+c_0R^2$. All derivatives of the cutoff lie outside the support of the space-time measure. The source integral is therefore zero, and the boundary integral is one.

\begin{proposition}[Localised Döblin--Fourier cancellation estimate]
\label{p:decay-loc}
There are $C,c>0$, depending only on $n,\lambda,\Lambda$, such that, if
$0\leq s<s+c_0R^2\leq1$, then
\[
    \sup_{v\in B_{\frac{3R}{4}}(v_0)}
    \norm{\widehat\omega_{s,v}^{D,\xi}}_{\TV}
    \leq C\exp\left(-c(R^3\abs\xi)^{\frac{1}{3}}\right),
    \qquad\xi\in\R^n.
    \refstepcounter{equation}\tag{\theequation}
    \label{e:local-decay}
\]
\end{proposition}

\begin{remark}
  The loss from $\frac{2}{3}$ to $\frac{1}{3}$ comes from the contribution of velocity-boundary data at arbitrarily short elapsed times.
\end{remark}

\begin{proof}
We first compare the evolution killed at the velocity boundary with the whole-space evolution. Their difference is the mass that has reached the boundary. For $v\in B_{\frac{3R}{4}}(v_0)$ and $0<h\leq c_0R^2$, we will show that
\[
    \norm{\widehat K^{D,\xi}_{s,s+h}(v,\cdot)}_{\TV}
    \leq C\left(e^{-\frac{cR^2}{h}}+e^{-ch\abs\xi^{\frac{2}{3}}}\right).
    \refstepcounter{equation}\tag{\theequation}
    \label{e:local-Dirichlet}
\]
The first term bounds the mass that has reached the boundary by time $s+h$; the second is the Fourier decay of the whole-space kernel. By scaling, it is enough to take $R=1$. We fix $t=s+h$ and, for $0\leq r\leq h$, write this lost mass as
\[
    q(r,v')=1-P^D_{t-r,t}(v',D).
\]
Then $q(0,\cdot)=0$ in $D$, $q=1$ on $\partial D$ for $r>0$, and
$\partial_rq-\mbox{Tr}(A(t-r,v')D_{v'}^2q)=0$. In particular, the mass of
$K_{s,s+h}-K^D_{s,s+h}$ is $q(h,v)$.
We set $\delta=(4\sqrt n)^{-1}$.  For $\alpha>0$,
\[
    \Psi(r,v')=\sum_{j=1}^n\sum_{\varepsilon=\pm1}
    \exp\left(\alpha\varepsilon(v'_j-v_j)-\alpha\delta
    +\Lambda\alpha^2r\right)
\]
satisfies
$(\partial_r-\mbox{Tr}(A(t-r,v')D_{v'}^2))\Psi\geq0$.  If
$v'\in\partial B_1(v_0)$ and $v\in B_{\frac{3}{4}}(v_0)$, then
$\varepsilon(v'_j-v_j)\geq\delta$ for some $j$ and $\varepsilon$.
Comparison, with $\alpha=\frac{\delta}{2\Lambda h}$, gives
\[
  q(h,v)\leq Ce^{-\frac{c}{h}}.
    \refstepcounter{equation}\tag{\theequation}
    \label{e:short-boundary}
\]
At scale $R$ the right-hand side is $Ce^{-\frac{cR^2}{h}}$.  Since
$0\leq K^D\leq K$, the difference is positive with this mass;
Proposition~\ref{p:decay} and the triangle inequality give
\eqref{e:local-Dirichlet}.

We now include the data on the velocity boundary by splitting $\omega^D_{s,v}$ at time $s+h$, with $0<h\leq c_0R^2$. The part that reaches the velocity boundary before $s+h$ has mass at most $Ce^{-\frac{cR^2}{h}}$. The remaining part first evolves inside $D$ according to $K^D_{s,s+h}$. From time $s+h$ it continues to the terminal face or the velocity boundary according to a non-negative kernel of mass one. This last evolution cannot increase total variation, so~\eqref{e:local-Dirichlet} gives
\[
  \norm{\widehat\omega_{s,v}^{D, \xi}}_{\TV}
  \leq C\left(e^{-\frac{cR^2}{h}}+e^{-ch\abs\xi^{\frac{2}{3}}}\right).
\]
For $R^3\abs\xi\geq1$ we choose $h=c_0R\abs\xi^{-\frac{1}{3}}$; for the other
frequencies we use the mass bound.  This proves~\eqref{e:local-decay}.
\end{proof}

\subsection{Differentiating the representation}

The Fourier decay just proved allows us to differentiate the boundary kernel in position as many times as needed. To estimate derivatives of the solution, we also need to integrate these kernel derivatives in position. This is possible because the velocity stays within distance $R$ of $v_0$ until exit, so the spatial displacement is bounded. We first record this support bound and then differentiate the representation formula.

For every unit vector $e$, the affine functions
\[
    \phi_\pm(r,x',w)
    =\pm e\cdot[x'-x-(r-s)v_0]-R(r-s)
\]
satisfy
\[
    \mathcal K\phi_\pm=\pm e\cdot(w-v_0)-R\leq0
    \quad\hbox{for }w\in D.
\]
Comparison, applied to non-negative increasing smooth functions of $\phi_\pm$ that vanish on $(-\infty,0]$,
shows that $\omega^D_{s,v}$ is supported in
\[
    \abs{x'-x-(r-s)v_0}\leq R(r-s)\leq c_0R^3.
    \refstepcounter{equation}\tag{\theequation}
    \label{e:boundary-support}
\]
With $y=x'-x-(r-s)v_0$ as before, the Fourier transform $\widehat\omega_{s,v}^{D,\xi}$ is still a finite complex measure in the time and velocity variables. Boundedness of the displacement makes it continuous in $\xi$ in total variation. Moreover, the decay in~\eqref{e:local-decay} is integrable even after multiplication by any power of $\abs\xi$. Fourier inversion therefore gives a density $\omega^D_{s,v}(y,\cdot)$ in displacement, taking values in finite measures of the remaining variables. Differentiating $m$ times under the Fourier integral gives, for $m\geq1$,
\[
    \sup_y\norm{D_y^m\omega^D_{s,v}(y,\cdot)}_{\TV}
    \leq C\int_{\R^n}\abs\xi^m
       e^{-c(R^3\abs\xi)^{\frac{1}{3}}}\dd\xi
    \leq C_mR^{-3(n+m)}.
\]
The projection of the support in~\eqref{e:boundary-support} onto the spatial variable has volume at most
$CR^{3n}$, hence
\[
    \int_{\R^n}
    \norm{D_y^m\omega^D_{s,v}(y,\cdot)}_{\TV}\dd y
    \leq C_mR^{-3m}.
\]
We now apply the representation to a solution of $\mathcal Lu=0$ in $Q_R^-(P_0)$. For $0<R'<R$, we reflect $(t,x)\mapsto(-t,-x)$, take $D=B_{R'}(v_0)$, and start the representation from each point in the image of $Q_{\frac{3R'}{4}}^-(P_0)$. Over the time interval of length $c_0(R')^2$, the time and spatial bounds are respectively $(\frac{9}{16}+\frac{1}{8})(R')^2<R^2$ and $(\frac{27}{64}+\frac{1}{8})(R')^3<R^3$, while $\abs{w-v_0}\leq R'<R$.

These bounds place every boundary point used in the representation strictly inside the original cylinder $Q_R^-(P_0)$ after reflecting back. We therefore use only interior values of $u$ and need no trace on $\partial Q_R^-(P_0)$. An $x$-cutoff equal to one near this support does not change the representation. Subtracting a constant, differentiating the kernel, and applying the preceding $L^1$ estimate gives, after reflecting back,
\[
    (R')^{3m}\norm{D_x^mu}_{L^\infty(Q_{\frac{3R'}{4}}^-(P_0))}
    \leq C_m\mathop{\rm osc}_{Q_R^-(P_0)}u,
    \qquad m\geq1,
\]
with $C_m$ depending only on $n,\lambda,\Lambda,m$, independently of $R'$. Letting $R'\uparrow R$ yields
\[
    R^{3m}\norm{D_x^mu}_{L^\infty(Q_{\frac{3R}{4}}^-(P_0))}
    \leq C_m\mathop{\rm osc}_{Q_R^-(P_0)}u,
    \qquad m\geq1.
    \refstepcounter{equation}\tag{\theequation}
    \label{e:x-gradient}
\]

\section{Hölder continuity}
\label{sec:holder}

The spatial gradient bound lets us recover regularity in time and velocity from the parabolic Krylov--Safonov theory. We move the transport term to the right-hand side, where it becomes $-v\cdot\nabla_xu$. On an interior cylinder, this gives a parabolic equation with a bounded source. Its Hölder estimate controls changes in time and velocity; the spatial gradient bound controls changes in position. Together they give Hölder continuity in all three variables.

There are $\alpha\in(0,1)$ and $C>0$, depending only on
$n,\lambda,\Lambda$, such that every bounded classical solution of
$\mathcal Lu=0$ in $Q_R^-(P_0)$, with $A=A(t,v)$, satisfies, for
$P=(t,x,v)$ and $P'=(s,y,w)$ in $Q_{\frac{R}{2}}^-(P_0)$,
\[
    \abs{u(P)-u(P')}
    \leq C\mathop{\rm osc}_{Q_R^-(P_0)}u
    \left(
    \frac{\abs{t-s}^{\frac{1}{2}}+\abs{v-w}
    +\abs{x-y-(t-s)v_0}^{\frac{1}{3}}}{R}
    \right)^\alpha.
    \refstepcounter{equation}\tag{\theequation}
    \label{e:holder}
\]
By a Galilean change of variables and kinetic scaling, it suffices to prove the estimate
for $P_0=0$ and $R=1$. If the oscillation is zero, the estimate is immediate. Otherwise, after subtracting a constant and dividing by the
oscillation, we may assume $\abs u\leq1$.  By~\eqref{e:x-gradient},
$\norm{\nabla_xu}_{L^\infty(Q_{\frac{3}{4}}^-(0))}\leq C$.  For each fixed
$x\in B_{\frac{1}{8}}$,
\[
    \partial_tu-\mbox{Tr}(A(t,v)D_v^2u)=-v\cdot\nabla_xu
    \quad\hbox{in }(-\frac{9}{16},0)\times B_{\frac{3}{4}},
\]
and the right-hand side is bounded by $C$.  The parabolic
Krylov--Safonov estimate~\cite{KrylovSafonov} at fixed $x$, combined with the spatial gradient bound, gives, for
$t,s\in(-\frac{1}{4},0)$, $v,w\in B_{\frac{1}{2}}$ and $x,y\in B_{\frac{1}{8}}$,
\[
    \begin{aligned}
    \abs{u(t,x,v)-u(s,y,w)}
    &\leq C\left\{
      \bigl(\abs{t-s}^{\frac{1}{2}}+\abs{v-w}\bigr)^\alpha
      +\abs{x-y}\right\}\\
    &\leq C\bigl(\abs{t-s}^{\frac{1}{2}}+\abs{v-w}
      +\abs{x-y}^{\frac{1}{3}}\bigr)^\alpha.
    \end{aligned}
\]
Undoing the normalisation and the changes of variables gives~\eqref{e:holder}.

\section{Autonomous coefficients in dimension one}
\label{s:spatial-roughness}

For $a=a(x,v)$, Fourier transformation in position no longer separates the frequencies. Away from zero velocity, we can recover this separation by using position as time. The resulting estimate stops when velocity leaves a prescribed interval. To obtain an estimate on a kinetic cylinder, we must also account for later returns to that interval.

We first bound these returns using velocity barriers. This gives the maximum principle at $p=6$, and hence Hölder continuity. We then propagate positive lower bounds along kinetic paths to obtain a Harnack comparison. This comparison gives better bounds on returns at later times, which improve the maximum principle to every $p>4$. We explain these steps below; their detailed proofs are in~\cite[Sections~8--9]{Companion}. We work first with smooth $a$, with constants independent of its derivatives, and use Section~\ref{s:remarks} for the approximation.

\subsection{Position as time}

After reflection, the operator is $\mathcal K=\partial_s+v\partial_X+a(X,v)\partial_{vv}$, with $a$ again denoting the reflected coefficient. On a bounded positive velocity interval separated from zero, division by $v$ gives
\[
    \frac{1}{v}\mathcal K
    =\partial_X+\frac{1}{v}\partial_s+\frac{a(X,v)}{v}\partial_{vv}.
\]
Here $X$ is the new time and $s$ is transported. The diffusion coefficient is independent of $s$, so we can Fourier transform in that variable. For terminal data prescribed at a fixed $X$ and depending only on velocity, with zero data at the velocity endpoints, the solution is independent of $s$. It therefore solves $\partial_X V+\frac{a}{v}\partial_{vv}V=0$. After reversing $X$, this is a uniformly parabolic equation. Thus the common-minorant argument uses classical parabolic Harnack in position-as-time and velocity, with the old time integrated out.

We centre the velocity interval at $\bar v\ne0$ and choose $r\leq\frac{|\bar v|}{2}$. We use
\[
    I_{\mathrm c}=\bar v+\left[-\frac{r}{4},\frac{r}{4}\right],
    \qquad I=\bar v+\left(-\frac{r}{2},\frac{r}{2}\right),
    \qquad I^*=\bar v+\left(-\frac{3r}{4},\frac{3r}{4}\right).
\]
The density is estimated on $I_{\mathrm c}$; $I$ and $I^*$ will specify when a visit begins and ends. The normalised change of variables treats both signs of $\bar v$. It gives~\eqref{e:intro-Lop} with uniform ellipticity and with the derivative of the transport rate bounded above and below by positive constants. The Fourier argument then gives an $L^q$ bound $C_qr^{\frac{6}{q}-4}$ for the space-time Green density killed at $\partial I^*$, for $1<q<\frac{3}{2}$~\cite[Lemma~8.2 and Proposition~8.3]{Companion}. The power comes from the kinetic volume $r^6$ and the source factor $r^2$.

\subsection{Counting visits}

We work at cylinder radius one and first suppose $|v_0|\leq8$. Near zero velocity we take $\bar v=\pm2r$, with $0<r\leq1$. We choose a bounded velocity interval of fixed length containing the closures of both the cylinder's velocity interval and $I^*$ in its interior. By autonomy, we shift the starting time to zero and write $P=(0,X,v)$ for the starting point in the reflected cylinder: $X$ is the initial position and $v$ the initial velocity. Starting with unit mass at $P$, we follow the positive evolution until the terminal time or departure from this larger interval. Its Green measure records the mass accumulated over time and dominates the measure on the original velocity strip. It is therefore enough to estimate this larger measure. The observation time is at most one.

A visit begins when velocity first reaches $\overline I$, and ends when it leaves $I^*$ or the observation time ends. After that departure, a new visit begins only when velocity reaches $\overline I$ again. Motion across $\partial I$ while still inside $I^*$ is part of the same visit. If the starting velocity lies in $\overline I$, the first visit begins immediately. Otherwise, it begins at the first arrival in $\overline I$.

More precisely, for a Borel set $E$ in time, position and velocity, $\gamma_k(E)$ is the amount of mass whose $k$th visit begins in $E$. We construct these measures by alternating evolution in $I^*$ with evolution in the part of the enlarged interval outside $\overline I$, using the exit measures as in~\cite[Lemma~8.6]{Companion}. In particular, $\gamma_1=\delta_P$ if the starting velocity lies in $\overline I$. Otherwise, $\gamma_1$ is the distribution of the first arrival there before observation ends. In either case its total mass is at most one. The sum of the $\gamma_k$ can have larger mass: a portion of the evolution is counted again whenever it leaves $I^*$ and returns to $\overline I$.

We divide elapsed time into $J_0=[0,r^2]$ and $J_j=(jr^2,(j+1)r^2]$ for $j\geq1$, and set
\[
    m_j=\sum_{k\geq1}\gamma_k(J_j\times\R_X\times\overline I).
\]
Thus $m_j$ counts, with their kernel weights, visits beginning during $J_j$. It does not measure the mass present in $I^*$ at one particular time.

The quadratic function $(v-v_-)(v_+-v)$, where $I^*=(v_-,v_+)$, bounds the mean duration before departure by $Cr^2$. Repeating this bound makes the mass still inside $I^*$ decrease exponentially on the time scale $r^2$. Contributions from visits beginning far apart in time therefore have little overlap. If $G_{P,I}$ denotes the density of the Green measure on the original velocity strip, starting from $P$ and restricted to $\{v\in I_{\mathrm c}\}$, the estimate for one visit gives
\[
    \norm{G_{P,I}}_{L^q}^{q}
    \leq C_qr^{6-4q}\sum_jm_j^q,
    \qquad 1<q<\frac{3}{2}.
    \refstepcounter{equation}\tag{\theequation}
    \label{e:auto-visits}
\]
The norm is in time, position and velocity. The precise summation is~\cite[Corollary~8.4]{Companion}: the exponentially decreasing contributions form a summable sequence. The dependence on $\sum_jm_j^q$ is what allows us to improve the estimate.

\subsection{The preliminary estimate}

The same velocity-barrier argument as in~\cite[Lemma~8.6]{Companion}, on the fixed enlarged strip, gives two bounds: $\sum_jm_j\leq\frac{C}{r}$ over the whole observation interval, and $m_j\leq C$ on each interval of length $r^2$. The first bound follows from the $Cr$ bound on time spent in $I^*$, divided by the factor $r^2$ supplied by the quadratic barrier. Applying the same barrier over a time interval of length $r^2$ gives the second bound. Consequently,
\[
    \sum_jm_j^q\leq(\sup_jm_j)^{q-1}\sum_jm_j\leq\frac{C_q}{r},
    \qquad
    \norm{G_{P,I}}_{L^q}\leq C_qr^{\frac{5}{q}-4}.
\]
In particular, at $q=\frac{6}{5}$ the density bound is $Cr^{\frac{1}{6}}$.

Every sufficiently small non-zero velocity lies in $I_{\mathrm c}$ for radii between two fixed multiples of its absolute value. Integration in $\frac{\dd r}{r}$ counts it with a fixed weight. Since
\[
    \int_0^1r^{\frac{1}{6}}\frac{\dd r}{r}<\infty,
\]
the complete Green density is in $L^{\frac{6}{5}}$, uniformly in $P$. Velocities bounded away from zero are covered by finitely many fixed-scale intervals, and cylinders with $|v_0|>8$ by a single interval; the Green measure gives no mass to $\{v=0\}$ \cite[Lemma~8.5 and Propositions~8.3, 8.9]{Companion}. Hölder's inequality against the source, followed by comparison, gives the maximum principle at $p=6$. This estimate is already strong enough to imply Hölder continuity for homogeneous solutions, by~\cite[Theorem~9.1, Corollary~9.10]{Companion}.

\subsection{From Hölder continuity to Harnack}

We apply Hölder continuity to a non-negative solution of $\mathcal Lu=0$ in the original variables. It turns a positive value into a positive lower bound on a neighbourhood. At a fixed interior scale, if the solution has value $k>0$ and nearby supremum $M$, it stays above $\frac{k}{2}$ on a kinetic cylinder of radius $c(\frac{k}{M})^{\frac{1}{\alpha}}$. Its velocity width is of that order, its time length is the square of that order, and its position width is the cube.

A moving barrier carries this lower bound along a kinetic path, whose position derivative is its velocity~\cite[Lemma~9.3]{Companion}. We also enlarge the cylinder as we move it. Doubling its radius costs only a fixed factor in the lower bound. Starting from radius $\ell$, only order $\log\frac{1}{\ell}$ doublings are needed to reach a fixed radius. The lower bound at a later point is therefore at least $c\ell^\beta k$, for a fixed $\beta>0$. This explains why propagation from the small neighbourhood remains useful.

The neighbourhood size still depends on $M$. Young's inequality lets us bound the supremum on an inner cylinder by a multiple of the later value plus an arbitrarily small multiple of the supremum on a larger cylinder. We repeat this estimate on nested cylinders. The small factor makes the last supremum disappear in the limit, and is chosen to offset the cost of the shrinking gaps between cylinders. After reflecting position, this gives the following comparison for the full-space autonomous evolution $\mathsf S_t$ with generator $v\partial_X+a(X,v)\partial_{vv}$:
\[
    \mathsf S_tF(X,v)\leq C\mathsf S_\tau F(X,v),
    \qquad 2t\leq\tau\leq3t,\quad |v|\leq\sqrt t,
    \quad F\geq0,
    \refstepcounter{equation}\tag{\theequation}
    \label{e:auto-harnack}
\]
for bounded Borel $F$~\cite[Lemma~8.7]{Companion}. Here $t$ is elapsed physical time. The constants and the Hölder exponent depend only on $\lambda,\Lambda$. This kinetic comparison follows from the preliminary $p=6$ estimate.

\subsection{Returns at later times}

For an indicator $\mathbf1_{I^*}$, the value $\mathsf S_t\mathbf1_{I^*}(X,v)$ is the mass in $I^*$ at time $t$, starting from $(X,v)$. We first bound its integral in time. The function $\psi_r(v)=\sqrt{r^2+v^2}$ has slope at most one and second derivative at least $\frac{c}{r}$ on $|v|\leq3r$. Its Green identity, together with the bound $2\Lambda T$ on the mean squared velocity change, gives
\[
    \int_0^T\mathsf S_t\mathbf1_{\{|v|\leq3r\}}(X,v)\dd t
    \leq Cr\sqrt T.
\]
This estimate alone allows a large amount of mass at a particular time. The Harnack comparison supplies the missing control. Our choice $\bar v=\pm2r$ gives $I^*\subset\{|v|\leq\frac{11r}{4}\}\subset\{|v|\leq3r\}$, while $t\geq9r^2$ implies $3r\leq\sqrt t$. Thus, for $|v|\leq3r$ and $t\geq9r^2$, averaging~\eqref{e:auto-harnack} over $\tau\in[2t,3t]$ gives
\[
    \mathsf S_t\mathbf1_{I^*}(X,v)
    \leq\frac{C}{t}\int_{2t}^{3t}\mathsf S_\tau\mathbf1_{\{|v|\leq3r\}}(X,v)\dd\tau
    \leq C\frac{r}{\sqrt t}.
\]
Together with the mass bound one for smaller times, this is $C(1+\frac{t}{r^2})^{-\frac{1}{2}}$.

We now convert mass present into the count of visits defined above. For the starting point $P=(0,X,v)$ fixed above, we first assume $v\in\overline I$. The quadratic barrier $\phi(v)=(v-v_-)(v_+-v)$ vanishes at departure from $I^*$ and is bounded below by $cr^2$ at the beginning of every visit. Also, $-\mathcal K\phi=2a$. We sum its Green identities over the portions of all visits lying in a later interval $J_j$, $j\geq1$, and discard the favourable contribution from mass already present at its left endpoint. This gives
\[
    m_j\leq C\left[
      \mathsf S_{(j+1)r^2}\mathbf1_{I^*}(X,v)
      +\frac{1}{r^2}\int_{jr^2}^{(j+1)r^2}\mathsf S_t\mathbf1_{I^*}(X,v)\dd t
    \right]
    \leq C(1+j)^{-\frac{1}{2}},\qquad j\geq1.
\]
The two terms account for mass still inside at the end and mass accumulated during the interval. Contributions at departure vanish because $\phi=0$ there. The initial visit is included separately in $m_0\leq C$.

If instead $v\notin\overline I$, the first arrival may be late. We apply the bound just proved from each arrival point, measuring time from that arrival, and integrate it against the first-arrival measure $\gamma_1$. Its total mass is at most one. Thus we average time translates of the same decreasing sequence. This preserves, up to a fixed factor for shifted time intervals, the bound on the sum of its $q$th powers:
\[
    \sum_jm_j^q
    \leq C_q\sum_{0\leq j\leq 2+r^{-2}}(1+j)^{-\frac{q}{2}}
    \leq C_qr^{q-2},
    \qquad 1<q<2.
\]
Using this in~\eqref{e:auto-visits} gives
\[
    \norm{G_{P,I}}_{L^q}^q\leq C_qr^{6-4q}r^{q-2}
    =C_qr^{4-3q},
    \qquad
    \norm{G_{P,I}}_{L^q}\leq C_qr^{\frac{4}{q}-3},
    \qquad 1<q<\frac{3}{2}.
\]
This is the improved interval estimate~\cite[Lemma~8.8]{Companion}, uniformly in the starting point. Its gain comes from controlling when visits begin, rather than only their total count.

\subsection{Summing near zero velocity}

We combine the interval estimates by the same integration in $\frac{\dd r}{r}$ used for the preliminary result. The improved bound is summable precisely when
\[
    \int_0^1 r^{\frac{4}{q}-3}\frac{\dd r}{r}<\infty
    \quad\Longleftrightarrow\quad q<\frac{4}{3}
    \quad\Longleftrightarrow\quad p=\frac{q}{q-1}>4.
\]
Velocities away from zero, including cylinders centred at large velocities, are covered by the estimate from the change of time variable. Restoring the cylinder radius gives the factor $R^{\frac{6}{q}-4}=R^{2-\frac{6}{p}}$~\cite[Proposition~8.9]{Companion}. Integration against the source and comparison prove~\eqref{eq:abp-kin}. This argument gives $p>4$. The Bellman improvement recorded after Theorem~\ref{thm:aleksandrov-sharp} also includes $p=4$. 

\section{Mixing estimates for rough spherical velocity diffusion}
\label{s:mixing}

We prove Theorem~\ref{thm:mixing} by constructing the transition kernel, estimating its Fourier modes, and then passing from total-variation mixing to the $L^2$ bound.

\subsection{The Krylov weak solutions}

In the previous sections, we have proved estimates for classical solutions. In this subsection, our primary interest lies in the properties of the evolution process rather than local regularity estimates, we also develop a framework for the evolution starting from measure-valued initial data. 

We use the velocity evolution supplied by Krylov's construction below and obtain position by integrating velocity. Its terminal-value solution operators preserve order and constants and satisfy comparison with smooth sub- and supersolutions. In the autonomous case, spatial reflection gives a semigroup of initial-value solutions to~\eqref{e:spherical-equation}, as explained in the spectral-gap proof. We call the solutions supplied by these operators \emph{Krylov solutions}. Comparison gives agreement with classical solutions whenever they exist with the same data. We do not assert uniqueness among all weak solutions.

We first construct the evolution of probability measures. In the autonomous case, densities relative to its invariant measure give the $L^2$ operators used in the spectral-gap proof.

For the adjoint equation, a solution is a family of probability measures $\rho_t$ on $\T^n\times\mathbb S^{n-1}$, continuous in time against continuous test functions, such that
\begin{equation}\label{e:spherical-weak}
  \int_{\mathbb T^n \times \mathbb S^{n-1}} \varphi\dd\rho_t-\int_{\mathbb T^n \times \mathbb S^{n-1}} \varphi\dd\rho_s
  =\int_s^t\int_{\mathbb T^n \times \mathbb S^{n-1}} \left(v\cdot\nabla_x\varphi
  +A(\tau,v):\nabla_S^2\varphi\right)\dd\rho_\tau\dd \tau
\end{equation}
for every smooth $\varphi(x,v)$ and $0\leq s\leq t$. All spatial integrals on the torus use normalised Haar measure.

We apply Krylov's construction and Markov selection~\cite{KrylovSelection,KrylovBook} to the velocity operator, using a finite family of sphere charts. In local coordinates it has the form
\[
  a^{ij}\partial_{ij}+b^k\partial_k,
  \qquad b^k=-a^{ij}\Gamma^k_{ij},
\]
where $\Gamma^k_{ij}$ are the Christoffel symbols. On the smaller chart domains, $a$ is uniformly elliptic and $b$ is bounded, with bounds depending only on the fixed atlas and ellipticity. No derivative of $A$ occurs. The local martingale problems describe the same intrinsic operator on overlaps; localisation and compactness of the sphere give a global evolution, to which the selection theorem applies.
This construction uses stochastic trajectories as an intermediate step. Smooth approximations produce velocity diffusions, and integrating $\dot X=V$ gives their position. The joint distributions solve the full kinetic equation above. To pass to measurable $A$, one uses the parabolic Aleksandrov estimate, which controls occupation against the coefficient approximation errors, so the limiting distributions satisfy the weak equation. We use only the resulting evolution of measures.

Krylov's Markov selection supplies a choice of evolution consistent with composition in time. We fix one such velocity evolution, time homogeneous when $A$ is autonomous, and obtain the kinetic evolution by integrating velocity. We do not claim uniqueness among all weak solutions. We use the kernel notation of Section~\ref{s.fourier}, now with velocities on $\mathbb S^{n-1}$. For fixed $s$ and $v$, the measure $K_{s,t}(v,\dd y,\dd w)$ is the solution at time $t$ supplied by the chosen evolution with initial measure $\delta_{(0,v)}$ at time $s$. Thus $v$ is the starting velocity, $w$ the terminal velocity, and $y$ the position displacement. We use measure notation because the argument does not require a density in velocity.

When a density in position exists, we write
\[
K_{s,t}(v,\dd y,\dd w) =K_{s,t}(y,v,\dd w)\dd y,
\]
where $\dd y$ is normalised Haar measure on $\T^n$. The expression on the right is a density in $y$ and remains a measure in $w$. As in Section~\ref{s.fourier}, the density notation places position displacement first.

On the torus, its velocity marginal is $P_{s,t}$ and
\[
\widehat K_{s,t}^\xi(v,\dd w)
=\int_{\T^n}e^{-i\xi\cdot y}K_{s,t}(v,\dd y,\dd w),
\qquad \xi\in\Z^n,
\qquad
|\widehat K_{s,t}^\xi|\leq P_{s,t},
\quad \widehat K_{s,t}^0=P_{s,t}.
\]
We denote the action on initial measures by $\mathsf U_{s,t}$. For every bounded Borel function $F$,
\[
\int F\dd(\mathsf U_{s,t}\rho)
=\int\left[\int F(x+y,w)K_{s,t}(v,\dd y,\dd w)\right]\rho(\dd x,\dd v).
\]
This action extends linearly to signed and complex measures.

We also use the corresponding kernel on $\R^n\times\mathbb S^{n-1}$, obtained from the same chosen velocity evolution, and denote it again by $K$. Projecting position onto $\T^n$ gives the torus kernel, so their Fourier transforms agree at every integer frequency. For the full-space kernel, the Fourier integral is over $\R^n$ and the frequency may be any vector in $\R^n$. We perform spatial rescaling on the full space before restricting to torus frequencies.

The selected evolution preserves positivity and mass. Stopping at velocity boundaries gives the Dirichlet kernels, and the stopped generator identity gives comparison with smooth barriers. Parabolic Harnack applies to nonnegative terminal-value solutions of the velocity evolution~\cite[Theorem~1.12 and its proof in Section~4]{Yastrzhembskiy2020}. It supplies the common lower bound~\eqref{e:parabolic-minoration} used below, without requiring classical regularity of $\rho_t$. All estimates are uniform over the choice of evolution; the invariant measure and generator refer to that choice.

\subsection{The Döblin--Fourier cancellation estimate}

We first suppose $\lambda=1$ and write $\kappa=\frac{\Lambda}{\lambda}$. We prove the estimate at every real frequency $\xi\ne0$. We set $e=\frac{\xi}{\abs\xi}$, fix the starting direction $v$, and write
\[
  e=av+b,\qquad b\perp v, \qquad a^2+|b|^2=1.
\]
For $\abs\xi\geq1$, we fix $0<r_0\leq\frac18$. We let $\bar r_0$ be the positive root of $|\xi|(|b|r^3+|a|r^4)=1$ and set $r=\min\{r_0,\bar r_0\}$. Then
\[
  \chi:=|\xi|(|b|r^3+|a|r^4)\in[r_0^4,1],
  \qquad r\leq|\xi|^{-\frac14}.
\]
Indeed, $|b|+\abs a\geq1$ and $r\leq1$ imply $|\xi|r^4\leq\chi\leq1$; if $r=r_0$, then $\chi\geq |\xi|r_0^4\geq r_0^4$; otherwise $\chi=1$.

We use normal coordinates $w=\operatorname{Exp}_v(rY)$ and time $t=s+r^2\tau$. On $\abs Y<4$ the rescaled velocity operator has ellipticity and drift bounds depending only on $n,\kappa$. The rate of change of the Fourier phase relative to the starting velocity is therefore
\[
  F(Y)=|\xi|r^2\bigl(e \cdot\operatorname{Exp}_v(rY)-a\bigr),
  \qquad
  e \cdot\operatorname{Exp}_v(rY)
  =a\cos(r\abs Y)+(b\cdot Y)\frac{\sin(r\abs Y)}{\abs Y}.
\]
The quotient extends smoothly at zero, and $F(0)=0$, $\abs{\nabla F}\leq C_*\chi\leq C_*$ on $\abs Y\leq2$. The scale therefore controls phase errors uniformly, including at the poles of $e\cdot w$. Figure~\ref{fig:spherical} shows why the poles determine the frequency power $\frac{1}{2}$.

\begin{figure}[t]
  \centering
\begin{tikzpicture}[>={Stealth[length=2.2mm]},lab/.style={font=\small},note/.style={font=\small}]
  \def\Rad{2.0}
  \def\arcr{25}     
  \begin{scope}
    \draw[thin] (0,0) circle (\Rad);
    \draw[dfgrey,thin,densely dotted] (-\Rad,0) -- (\Rad,0);
    \draw[->,thick] (0,0) -- (0,\Rad+0.75) node[above,lab] {$e=\frac{\xi}{\abs\xi}$};
    \draw[dfone,line width=2pt,-{Stealth[length=3mm]}] (0:\Rad) arc (0:-\arcr:\Rad);
    \draw[dftwo,line width=2pt,-{Stealth[length=3mm]}] (90:\Rad) arc (90:90-\arcr:\Rad);
    \fill (0:\Rad) circle (1.6pt) node[left,note] {$v$};
    \fill (90:\Rad) circle (1.6pt) node[below=1pt,note] {$v$};
    \node[note,dfone,right] at (-\arcr/2:\Rad+0.08) {$r$};
    \node[note,dftwo,above right=-1pt] at (90-\arcr/2:\Rad+0.05) {$r$};
    \node[note,dfone,anchor=west] at (-\Rad+0.1,-0.45) {equator, $v\perp e$};
    \node[note,dftwo,anchor=east] at (-0.45,\Rad+0.2) {pole, $v=e$};
    \def\X{3.15}
    \draw[->] (\X,-\Rad-0.15) -- (\X,\Rad+0.45) node[above,lab] {$e\cdot w$};
    \foreach \y/\l in {-\Rad/-1, \Rad/1}
      \draw (\X,\y) -- ++(-0.06,0) node[left,note] {$\l$};
    \draw[dfone,thin,densely dashed] (0:\Rad) -- (\X,0);
    \draw[dfone,thin,densely dashed] (-\arcr:\Rad) -- (\X,{-\Rad*sin(\arcr)});
    \draw[dftwo,thin,densely dashed] (90:\Rad) -- (\X,\Rad);
    \draw[dftwo,thin,densely dashed] (90-\arcr:\Rad) -- (\X,{\Rad*cos(\arcr)});
    \draw[dfone,line width=2.2pt] (\X,{-\Rad*sin(\arcr)}) -- (\X,0);
    \draw[dftwo,line width=2.2pt] (\X,{\Rad*cos(\arcr)}) -- (\X,\Rad);
    \node[note,dfone,right] at (\X+0.08,{-\Rad*sin(\arcr)/2}) {$\sin r\approx r$};
    \node[note,dftwo,right] at (\X+0.08,{\Rad*(1+cos(\arcr))/2-0.02}) {$1-\cos r\approx\frac{r^2}{2}$};
    \node[lab] at (-\Rad-0.35,\Rad+0.75) {(a)};
  \end{scope}

  \begin{scope}[shift={(6.9cm,-2.1cm)},x=0.36cm,y=0.5cm]
    \begin{scope}[yshift=4.5cm]
      \draw[dfone,thick] plot file {\figdata f3_equator.dat};
      \draw[dfone!70,thin] plot file {\figdata f3_th1.dat};
      \draw[dfone!40,thin] plot file {\figdata f3_th2.dat};
      \draw[dftwo,thick] plot file {\figdata f3_pole.dat};
    \end{scope}
    \draw[->] (0,0) -- (12.6,0);
    \node[lab] at (6,-1.35) {frequency $q=\abs\xi$};
    \draw[->] (0,0) -- (0,8.6);
    \node[lab,rotate=90] at (-3.4,4.3) {time scale $r^2$};
    \foreach \x/\l in {0/10^0, 4/10^4, 8/10^8, 12/10^{12}}
      \draw (\x,0) -- ++(0,-0.12) node[below,note] {$\l$};
    \foreach \y/\l in {1/10^{-8}, 4/10^{-5}, 7/10^{-2}}
      \draw (0,\y) -- ++(-0.12,0) node[left,note] {$\l$};
    \node[note,dftwo,right] at (8.6,5.2) {pole: $q^{-\frac12}$};
    \node[note,dfone,right] at (8.4,0.6) {equator: $q^{-\frac23}$};
    \node[note,dfone!70,align=left] at (5.6,1.7) {$\theta=0.1,\ 0.01$};
    \node[lab] at (-2.2,9.5) {(b)};
  \end{scope}
\end{tikzpicture}
  \caption{\small{Spherical velocities, with $\lambda=1$, $q=\abs\xi$ and $e=\frac{\xi}{q}=av+b$, $b\perp v$: a velocity $w$ moves the phase $\xi\cdot x$ at speed $q\,e\cdot w$. \textbf{(a)} On a great circle through $v$ and $e$, an excursion of length $r$ changes $e\cdot w$ by $\sin r\approx r$ at the equator, but only by $1-\cos r\approx\frac{r^2}{2}$ at the pole. Held for a time of order $r^2$, it builds a relative phase of order $qr^3$, respectively $qr^4$. \textbf{(b)} The time scale $r^2$ for $r=\min\{r_0,\bar r_0\}$, where $\bar r_0$ is the positive root of $q(\abs br^3+\abs ar^4)=1$ and $r_0=\frac18$, for several angles $\theta$ between $v$ and $e$ (so $\abs a=\cos\theta$, $\abs b=\sin\theta$). In the uncapped, high-frequency regime it is $q^{-\frac23}$ at the equator and $q^{-\frac12}$ at the pole; the horizontal plateau at low frequencies comes from imposing $r\leq r_0$. The pole sets the rate $q^{\frac12}$ in~\eqref{e:spherical-DF}, which is optimal.}}
  \label{fig:spherical}
\end{figure}
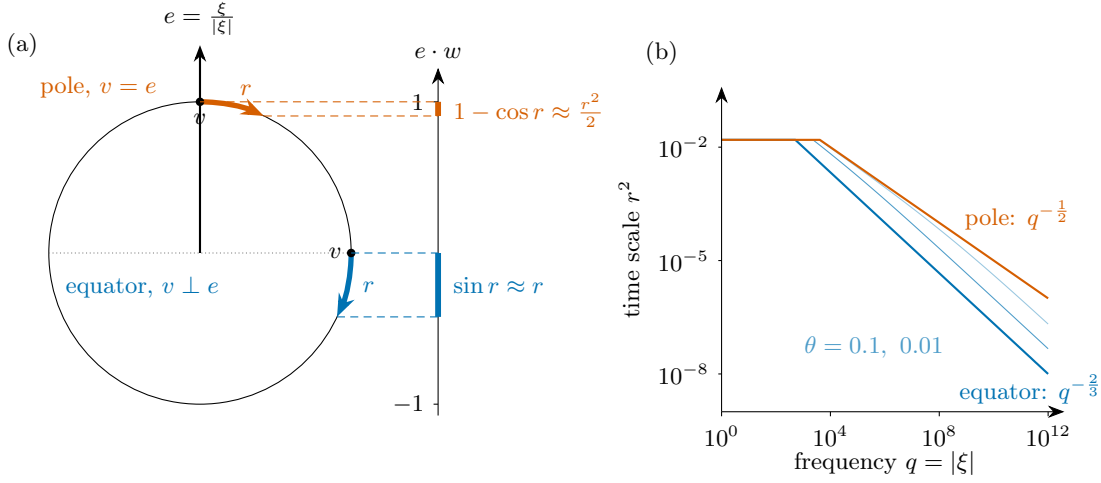

When $b\ne0$, we choose a unit tangent vector $u=-\frac{b}{\abs b}$ if $a\geq0$, and $u=\frac{b}{\abs b}$ if $a<0$. When $b=0$, any unit tangent vector will do. The linear and quadratic phase changes then have the same sign:
\[
  \abs{F(zu)}=|\xi| r^2\left(|b|\sin(rz)+\abs a(1-\cos(rz))\right),
  \qquad 0\leq z\leq1.
\]
In particular, $\abs{F(u)}$ lies between $\frac\chi3$ and $\chi$. We keep one path at zero. The other travels to $u$ in unit time, stays there for time $H$, and returns in unit time. With
\[
  J=2\int_0^1\abs{F(zu)}\dd z,\qquad
  H=\frac{\pi-J}{\abs{F(u)}},\qquad
  L=2+H\leq L_*:=2+\frac{3\pi}{r_0^4},
\]
the accumulated phases differ by $\pi$. Here $J\leq2\abs{F(u)}\leq2<\pi$, so $H>0$, and both paths have speed at most one.

The tube argument of Section~\ref{s.fourier} also applies with bounded chart drift. For a tube of radius $\eta$ and centre $\gamma$, the same function $\psi=(1-\frac{\abs{Y-\gamma}^2}{\eta^2})_+^N$ satisfies
\[
  (\partial_\tau+a^{ij}\partial_{ij}+b^i\partial_i)\psi
  \geq-C\bigl(\eta^{-2}+(B+\abs{\gamma'}_\infty)^2\bigr)\psi,
\]
where $B$ bounds the drift. The maximum principle gives a mass $m_*>0$ in each Dirichlet kernel along the two paths, uniformly in $|\xi|,v,s$. We choose fixed $0<\eta<\rho<\frac18$ so that
\[
  \varepsilon=C_*\eta L_*+16C_*\rho^3<\frac14.
\]
We scale each tube measure to mass $\frac{m_*}{2}$, using a factor at most $\frac{1}{2}$, so that their sum is dominated by the full kernel. We continue both for time $4\rho^2$ in the velocity ball $B_{4\rho}(0)$. Parabolic Harnack, followed by the constant-path barrier, gives a positive measure $\Theta$ below both terminal velocity marginals, of mass at least $\theta_*>0$. Multiplication by $\frac{\dd\Theta}{\dd\nu_i}$, where $\nu_i$ are those marginals, gives measures $N_i$ with velocity marginal $\Theta$ and sum dominated by the full kernel. Their phases lie within $\varepsilon$ of two values differing by $\pi$: the errors from the tubes and common continuation are bounded by the two terms defining $\varepsilon$.

The cancellation calculation of Section~\ref{s.fourier} now gives
\[
  \norm{\widehat K_{s,s+r^2(L+4\rho^2)}^\xi(v,\cdot)}_{\TV}
  \leq1-2(1-\varepsilon)\theta_*<1.
\]
We append unrestricted evolution up to the common time $s+(L_*+4\rho^2)\abs\xi^{-\frac{1}{2}}$. This cannot increase total variation. Although the shorter construction time depends on $v$, the resulting contraction is at a common time and is uniform in $v$. Iteration therefore gives $Ce^{-c(t-s)\abs\xi^{\frac{1}{2}}}$ for $\abs\xi\geq1$.

For $0<\abs\xi\leq1$, we use the same paths constructed at reference frequency $e$, retaining portions with velocity marginal $\Theta$ as above. Their phase measures have supports within $\varepsilon$ of values differing by $\pi$. At frequency $\xi$, each phase is multiplied by $\abs\xi$. For any two phases in these supports, their absolute difference $\Delta$ lies in $[\pi-2\varepsilon,\pi+2\varepsilon]$, and
\[
  \left|e^{-i\abs\xi\phi_1}+e^{-i\abs\xi\phi_2}\right|
  =2\sqrt{1-\sin^2\left(\frac{\abs\xi\Delta}{2}\right)}
  \leq2-\abs\xi^2\cos^2\varepsilon.
\]
Here concavity of sine gives $\sin(\frac{\abs\xi\Delta}{2})\geq\abs\xi\sin(\frac{\Delta}{2})\geq\abs\xi\cos\varepsilon$. We disintegrate the two phase measures against $\Theta$ and integrate this inequality against the product of their conditional probabilities. The positive remainder then gives contraction by at least $\theta_*\abs\xi^2\cos^2\varepsilon$. Appending evolution to time $s+L_*+4\rho^2$ and iterating proves
\[
  \sup_v\norm{\widehat K_{s,t}^\xi(v,\cdot)}_{\TV}
  \leq Ce^{-c(t-s)\min\{\abs\xi^2,\abs\xi^{\frac{1}{2}}\}}.
\]
For general $\lambda$, we set $\tau=\lambda t$ and $X=\lambda x$ on the lifted space. The matrix becomes $\lambda^{-1}A(\frac{\tau}{\lambda},v)$, with bounds $1,\frac{\Lambda}{\lambda}$, and the frequency becomes $\frac{\xi}{\lambda}$. Undoing this change gives~\eqref{e:spherical-DF}. The change in position is made on the lift, before restricting to torus frequencies.

The frequency power $\frac{1}{2}$ is optimal in this uniform estimate. We start at a pole $v=e$ and set $\varphi(w)=1-e\cdot w$. Since $\nabla_S^2(e\cdot w)=-(e\cdot w)g_S$, where $g_S$ is the sphere metric, testing the velocity equation gives
\[
  \int_{\mathbb S^{n-1}}\varphi(w)P_{s,s+\tau}(e,\dd w)
  \leq(n-1)\Lambda\tau.
\]
At frequency $Re$, we remove the constant phase $e^{-iR\tau}$ and test the Fourier equation against $1$. The remaining phase speed is $R\varphi$, so domination by $P$ yields
\[
  \norm{\widehat K_{s,s+\tau}^{Re}(e,\cdot)}_{\TV}
  \geq1-\frac{n-1}{2}\Lambda R\tau^2.
\]
We take $e=(1,0,\dots,0)$ and let integers $R$ tend to infinity. A uniform bound $Ce^{-c\tau R^\alpha}$ with $\alpha>\frac{1}{2}$ would contradict this inequality at $\tau=R^{-\beta}$, where $\frac{1}{2}<\beta<\alpha$.

\subsection{Spatial regularity and mixing in total variation}

On the torus, the Fourier estimate gives spatial regularity directly. We set $\beta_\lambda=\min\{\lambda,\lambda^{-1}\}$. For every $k\in\Z^n\setminus\{0\}$,
\[
\min\left\{\frac{\abs k^2}{\lambda},
\lambda^{\frac{1}{2}}\abs k^{\frac{1}{2}}\right\}
\geq\beta_\lambda\abs k^{\frac{1}{2}}.
\]
With respect to normalised Haar measure $\!\dd y$, the spatial density of the torus transition measure is
\[
K_{s,t}(y,v,\dd w)
=
P_{s,t}(v,\dd w)
+
\sum_{k\ne0}e^{ik\cdot y}\widehat K_{s,t}^k(v,\dd w).
\]
The series and all its spatial derivatives converge uniformly in total variation. For a multi-index $\alpha$ of order $m=\abs\alpha$ and $0<\beta_\lambda(t-s)\leq1$, the Fourier bound gives
\begin{equation*}
  \sup_{y,v}\norm{\partial_y^\alpha K_{s,t}(y,v,\cdot)}_{\TV}
  \leq C^{m+1}(m!)^2[\beta_\lambda(t-s)]^{-2m-2n}.
\end{equation*}
This proves spatial Gevrey regularity of order two, with values in finite velocity measures. Integration against any initial probability gives the same bound; no density in velocity is assumed.

We next control the zero spatial mode. After normalising time by $\lambda$, a finite sphere covering and parabolic Harnack chains give constants $T_0>\tau_0>0$ and $\varepsilon_0\in(0,1)$, depending only on $n,\frac{\Lambda}{\lambda}$, such that
\[
P_{s,s+\frac{T_0}{\lambda}}(v,\cdot)\geq\varepsilon_0\nu_s,
\qquad
\nu_s=P_{s+\frac{T_0-\tau_0}{\lambda},
s+\frac{T_0}{\lambda}}(v_*,\cdot).
\]
Here $v_*$ is fixed and $\nu_s$ is a probability. To obtain uniform constants, we join $v$ to $v_*$ by a geodesic of length at most $\pi$, cover it by a bounded number of chart balls, and allocate a fixed Harnack time to each link. Subtracting the common lower bound and iterating gives, for every mass-zero signed or complex measure $\alpha$,
\begin{equation}
  \label{e:spherical-zero}
  \norm{\int_{\mathbb S^{n-1}} P_{s,t}(v,\cdot) \dd \alpha(v)}_{\TV}
  \leq q_0^{\lfloor\frac{\lambda(t-s)}{T_0}\rfloor}\norm\alpha_{\TV},
  \qquad q_0=1-\varepsilon_0<1.
\end{equation}
For $\beta_\lambda(t-s)\geq1$, summing the non-zero Fourier modes gives
\[
  \norm{\mathsf U_{s,t}\delta_{(x,v)}
    -\dd y\otimes P_{s,t}(v,\cdot)}_{\TV}
  \leq C\sum_{k\ne0}e^{-c\beta_\lambda(t-s)\abs k^{\frac12}}
  \leq Ce^{-\frac{c\beta_\lambda}{2}(t-s)}.
\]
Combining this with~\eqref{e:spherical-zero}, and integrating over the starting point, proves
\[
  \norm{\mathsf U_{s,t}\rho-\pi_t}_{\TV}
  \leq Ce^{-c\beta_\lambda(t-s)}
\]
for every initial probability $\rho$ and every velocity probability $\mu_s$, where $\mu_t=\mu_sP_{s,t}$ and $\pi_r=\dd x\otimes\mu_r$ for $r=s,t$. Translation invariance in position gives $\mathsf U_{s,t}\pi_s=\pi_t$. For $\beta_\lambda(t-s)<1$, both measures are probabilities, so their total-variation distance is at most two. Increasing $C$ therefore extends the estimate to all $s<t$. All constants depend only on $n,\frac{\Lambda}{\lambda}$.

\subsection{The spectral gap estimate}

We first estimate the evolution of densities relative to $\pi_s$ and $\pi_t$. Since $\mathsf U_{s,t}\pi_s=\pi_t$, positivity gives $\mathsf U_{s,t}(f\pi_s)\ll\pi_t$ for $f\in L^1(\pi_s)$. We define $S_{s,t}f$ by
\[
  \mathsf U_{s,t}(f\pi_s)=(S_{s,t}f)\pi_t,
  \qquad
  \Pi_s f:=\int f\dd\pi_s,
\]
where $\Pi_s f$ is viewed as a constant function. The operator $S_{s,t}$ preserves the integral and satisfies $S_{s,t}1=1$. Integrating the total-variation mixing bound against $f\pi_s$ gives the first estimate below, and positivity gives the second:
\[
  \begin{aligned}
  \norm{S_{s,t}-\Pi_s}_{L^1(\pi_s)\to L^1(\pi_t)}
    &\leq Ce^{-c\beta_\lambda(t-s)},\\
  \norm{S_{s,t}-\Pi_s}_{L^\infty(\pi_s)\to L^\infty(\pi_t)}
    &\leq2.
  \end{aligned}
\]
Interpolation therefore gives, with $m=\Pi_s f$,
\[
  \norm{S_{s,t}f-m}_{L^2(\pi_t)}
  \leq Ce^{-\frac{c}{2}\beta_\lambda(t-s)}
       \norm{f-m}_{L^2(\pi_s)}.
\]

For autonomous $A = A(v)$, the chosen evolution is time homogeneous. By~\eqref{e:spherical-zero}, the velocity evolution over time $\frac{T_0}{\lambda}$ is a strict contraction on the complete space of probability measures with the total-variation distance. Its unique fixed point $\mu$ is invariant at every time, since the evolution operators commute. Thus $\pi=\dd x\otimes\mu$ is stationary. Taking $\rho_\tau=\pi$ and a test function depending only on velocity in~\eqref{e:spherical-weak}, and dividing by $t-s>0$, gives the stationary identity~\eqref{e:spherical-invariant}.

We set $S_t=S_{0,t}$ with reference measure $\pi$. Positivity and preservation of the integral make $S_t$ a contraction on $L^2(\pi)$. Weak continuity of the measure solutions, followed by approximation of $L^2(\pi)$ test functions by smooth functions, gives weak continuity of $S_t$. The contraction property then gives strong continuity.

For this stationary reference measure, the Hilbert-space adjoint of $S_t$ acts on terminal functions through the same kernel:
\[
(S_t^*F)(x,v)
=\int F(x+y,w)K_{0,t}(v,\dd y,\dd w).
\]
The generator of $S_t^*$ acts on smooth functions as $v\cdot\nabla_x+A:\nabla_S^2$.
Spatial reflection, $(\mathsf Rf)(x,v)=f(-x,v)$, preserves $\pi$ and reverses the transport term. Consequently the solution of~\eqref{e:spherical-equation} associated with the chosen evolution is
\[
  u(t)=\mathsf R S_{t-s}^*\mathsf R\,u(s).
\]
Taking the adjoint and applying this reflection preserve the $L^2$ decay bound on the mean-zero subspace. This proves~\eqref{e:spherical-L2-gap}, after decreasing $c$.

On the subspace $\int f\dd\pi=0$, the generator $\mathcal G$ of $S_t$ has resolvent
\[
  (z-\mathcal G)^{-1}f
  =\int_0^\infty e^{-zt}S_tf\dd t,
  \qquad \operatorname{Re}z>-c\beta_\lambda.
\]
The decay estimate makes this integral convergent. The adjoint generator has the conjugate spectrum, and reflection is a unitary conjugation. Thus the generator of the solution semigroup above also has spectrum in $\{\operatorname{Re}z\leq-c\beta_\lambda\}$. This proves the spectral gap.

\section{Remarks and counter-examples}
\label{s:remarks}

\subsection{Measurable coefficients}

For smooth coefficients, approximation of $(\mathcal Lu)_+$ from above on inner cylinders gives the estimate for $L^p$ sources. For measurable coefficients, we mollify $A$ to $A_j$ in the variables on which it depends, preserving ellipticity, and keep $u$ fixed. Its positive source is bounded by
\[
    f_++\abs{(A_j-A):D_v^2u}.
\]
The last term tends to zero in $L^p$ on compact inner cylinders because $D_v^2u$ is locally bounded. The estimates are independent of coefficient derivatives, so we first let $j\to\infty$, then exhaust the original cylinder. Continuity of $u$ handles the kinetic boundary suprema; the radii and source norms also converge. This proves~\eqref{eq:abp-kin} without a global Hessian bound; see~\cite[Section~6]{Companion} for the inner cylinders and their boundary convergence.

For the spatial and Hölder estimates with $A=A(t,v)$, we use reflected coordinates and a smooth cutoff $\chi$, compactly supported in the interior and equal to one on the inner cylinder. With $U$ denoting the reflected solution, we set
\[
F_j=\chi\,\mathcal K_{A_j}U=\chi(A_j-A):D_v^2U
\quad\text{almost everywhere}.
\]
Then $\norm{F_j}_{L^p}\to0$ for fixed $p>2n+1$. The whole-space solution of $\mathcal K_{A_j}W_j=-F_j$, with zero data at a fixed terminal time after the time support of $F_j$, satisfies $\norm{W_j}_\infty\to0$ on the finite time interval under consideration by~\eqref{e:Green-density}. Thus $U+W_j$ is homogeneous and smooth on the inner cylinder~\cite{Hormander1967}, and converges uniformly to $U$. The spatial derivative bounds pass to the limit in distributions, and the Hölder bound passes directly. Interior exhaustion, with the top time kept below $t_0$ until the last limit, completes the argument.

\subsection{Restricting the source to the positive set}

Only the source where $u$ is positive is needed in the maximum principle. The same comparison, applied to smooth convex approximations $\phi_\varepsilon$ of the positive part with $0\leq\phi_\varepsilon'\leq1$ and $\phi_\varepsilon'=0$ on $(-\infty,0]$, gives~\eqref{eq:abp-kin} with source norm
\[
    \norm{f_+}_{L^p(Q^-(P_0)\cap\{u>0\})}.
\]
On $Q_R^-(P_0)$, kinetic scaling gives the source term
\[
    C_{n,\lambda,\Lambda,p}R^{2-\frac{4n+2}{p}}
    \norm{f_+}_{L^p(Q_R^-(P_0)\cap\{u>0\})}.
\]
Indeed, $\mathcal L\phi_\varepsilon(u)\leq\phi_\varepsilon'(u)f$, so the positive source vanishes outside $\{u>0\}$; we now let $\varepsilon$ decrease to zero. For measurable coefficients, we apply this chain rule first with $A_j$ and use the smooth-coefficient estimate before passing to the limit. The non-positive gradient-square term is discarded, so no extra integrability of $D_v^2\phi_\varepsilon(u)$ is required.

\subsection{The Harnack comparison}
We use the full-space comparison of~\cite[Lemma~8.7]{Companion}, whose proof combines H\"older continuity with propagation of positive lower bounds. It does not apply to an evolution killed on a velocity interval of one sign: such paths cannot return to the same position at a later time. General local Harnack and weak Harnack estimates need further arguments; see~\cite{KrylovSafonov} for the parabolic theory and~\cite[Theorem~A.1 and Remark~12]{GuerandImbert} for kinetic covering.

\subsection{Decay, scaling and comparison with Kolmogorov's fundamental solution}

The exponential decay in~\eqref{eq:DF} comes from repeating a fixed contraction. Kinetic scaling determines the combination $(t-s)\lambda^{\frac{1}{3}}\abs\xi^{\frac{2}{3}}$, while iteration gives an exponential that is linear in this quantity.

For constant $A$, Kolmogorov's fundamental solution gives the exact formula
\[
  \widehat K_{s,t}^{\xi}(v,\dd w)
  =e^{-i\frac{(t-s)\xi\cdot(v+w)}{2}}
   e^{-\frac{(t-s)^3\xi\cdot A\xi}{12}}
   P_{s,t}(v,\dd w),
\]
and hence
\[
  \norm{\widehat K_{s,t}^{\xi}(v,\cdot)}_{\TV}
  =\exp\left(-\frac{(t-s)^3}{12}\xi\cdot A\xi\right).
\]
Its exponent is comparable to $\lambda(t-s)^3\abs\xi^2$, with comparison constants depending only on $\frac{\Lambda}{\lambda}$. Thus both estimates become effective on the time scale $\lambda^{-\frac{1}{3}}\abs\xi^{-\frac{2}{3}}$, although the constant-coefficient kernel decays faster beyond that scale. Proposition~\ref{p:decay} preserves this time scale for measurable coefficients; we do not claim that its dependence on elapsed time is optimal.

The exponent $\frac{1}{3}$ in the localised estimate of Proposition~\ref{p:decay-loc} comes from balancing the two terms in~\eqref{e:local-Dirichlet} and is not asserted to be optimal.

\subsection{Counter-examples}
\label{ss:scaling-obstructions}

The constructions below keep positive height as the source tends to zero, or make the height grow with a bounded source.

For the elliptic and parabolic scaling bounds, we shrink a smooth bump of fixed height to spatial scale $r$, using time scale $r^2$ in the parabolic case. The source norms have orders $r^{\frac{n}{p}-2}$ and $r^{\frac{n+2}{p}-2}$, respectively. This excludes $p<\frac{n}{2}$ and $p<\frac{n+2}{2}$.

The additional radial obstruction comes from Pucci's example. For $n\geq2$ and $a=1-(n-1)\kappa^{-1}>0$, the function $1-|x|^a$ solves the homogeneous elliptic equation away from zero for a matrix with radial eigenvalue $\Lambda$ and tangential eigenvalues $\lambda$. Smoothing inside $B_\varepsilon$ gives a source of size $O(\varepsilon^{a-2})$ and hence an $L^p$ norm of order $\varepsilon^{\frac{n}{p}+a-2}$. This excludes $p<\frac{n\kappa}{\kappa+n-1}$; the remaining parameter range is covered by the scaling bound. To obtain the parabolic obstruction with zero initial data, we subtract the homogeneous Dirichlet evolution of the smoothed profile from that profile. The source is unchanged and independent of time. Uniform decay of the evolution leaves a positive height at a fixed time. Thus the same radial lower bound also applies in the parabolic case.

For the kinetic endpoint $p=1+2n$, a single rescaled bump has a source norm independent of scale. Instead, we choose $u_0$ smooth up to $t=0$, equal to one on $Q_{\frac{1}{2}}^-(0)$ and zero near the initial and lateral boundary of $Q_1^-(0)$. We extend it by zero only across those faces and set
\[
  u_N(t,x,v)=\sum_{k=1}^N\frac{1}{k} u_0\left(2^{2k}t,2^{3k}x,2^kv\right).
\]
These finite sums are smooth, vanish on $\partial_kQ_1^-(0)$ and have height $u_N(0,0,0)=\sum_{k=1}^N k^{-1}$. Their sources lie in disjoint kinetic annuli. At the critical exponent, each summand contributes $Ck^{-p}$ to $\norm{\mathcal Lu_N}_{L^p}^p$. The source norms therefore remain bounded while the heights diverge. This proves failure of the strong Lebesgue estimate at the endpoint, even for constant coefficients. The time integral in the proof for $A(t,v)$ diverges logarithmically at $p=1+2n$. For smooth coefficients, the argument admits a Lorentz-space refinement with the source measured in $L^{1+2n,1}$; see~\cite[Remark~5.4]{Companion} for the weak Green-density estimate that gives this endpoint.

For autonomous kinetic coefficients, the obstruction is stronger. The companion paper~\cite[Appendix~C]{Companion} gives a stationary construction: a positive profile $H$, vanishing only at the origin, solves $A(x,v):D_v^2H-v\cdot\nabla_xH=0$ almost everywhere away from it for a suitable autonomous uniformly elliptic field $A$, and satisfies $H(r^3x,rv)=r^\alpha H(x,v)$ for $0<\alpha<1$. The profile $1-H$ has positive height and zero boundary values on $\{H<1\}$. Flattening it near the origin confines its source to a kinetic annulus of volume $O(\varepsilon^{4n})$, with source size $O(\varepsilon^{\alpha-2})$. Hence
\[
    \norm{f_+}_{L^p}=O\bigl(\varepsilon^{\alpha-2+\frac{4n}{p}}\bigr).
\]

The stationary profile alone does not contradict the time-dependent estimate, because its height is already present in the initial data. As explained in~\cite[Appendix~C]{Companion}, subtracting a decaying super-solution and taking a smooth approximation of the positive part produces zero initial and lateral values, while retaining a fixed positive height at a fixed later time. The positive source satisfies the same bound, and the time interval is independent of $\varepsilon$. For each fixed $\alpha$, the ellipticity bounds also stay fixed as $\varepsilon\to0$. The source norm tends to zero when $p<\frac{4n}{2-\alpha}$. Taking $\alpha$ close to one excludes every $p<4n$ for a suitable ellipticity ratio; it does not assert failure below $4n$ for each fixed ratio.

\end{document}